\documentclass[11pt, reqno, oneside]{amsart}

\usepackage{amssymb, mathrsfs, verbatim, graphicx,graphics, epsfig, psfrag, mathtools, bm, bbm, etex}

\usepackage{mathabx}
\usepackage{color, xcolor}
\usepackage{microtype}

\usepackage{abstract}
\usepackage[all]{xy}

\usepackage[left= 3.4 cm, right= 3.4 cm, top= 2.8 cm, bottom=2.7 cm, foot= 1.2 cm, marginparwidth=2.7 cm, marginparsep=0.5 cm]{geometry}

\usepackage[inline]{enumitem}

\usepackage{hyperref}
\hypersetup{colorlinks}
\definecolor{darkred}{rgb}{0.5,0,0}
\definecolor{darkgreen}{rgb}{0,0.5,0}
\definecolor{darkblue}{rgb}{0,0,0.5}
\hypersetup{colorlinks, linkcolor=darkblue, filecolor=darkgreen, urlcolor=darkred, citecolor=darkblue}
\makeatletter 
\@addtoreset{equation}{section}
\makeatother  

\numberwithin{equation}{section}

\usepackage{tikz-cd}
\usepackage{tikz}
\usetikzlibrary{arrows}
\usepackage{float}

\usetikzlibrary{decorations.pathreplacing}

\newtheorem{thm}{Theorem}[section]
\newtheorem{cor}[thm]{Corollary}

\newtheorem{conj}[thm]{Conjecture}
\newtheorem{prop}[thm]{Proposition}
\newtheorem{lemma}[thm]{Lemma}
\theoremstyle{definition}
\newtheorem{defn}[thm]{Definition}
\theoremstyle{remark}
\newtheorem{rem}[thm]{Remark}
\newtheorem{hyp}[thm]{Hypothesis}

\newtheorem{notation}[thm]{Notation}
\newtheorem{convention}[thm]{Convention}

\newcounter{notes}
{\end{list}}

\makeatletter
\@addtoreset{thm}{section}
\makeatother

\newcommand\qu{/\kern-.7ex/} 
\renewcommand{\setminus}{\smallsetminus}

\newcommand{\Gammait}{{\mathit{\Gamma}}}

\newcommand{\beq}{\begin{equation}}
\newcommand{\eeq}{\end{equation}}
\newcommand{\beqn}{\begin{equation*}}
\newcommand{\eeqn}{\end{equation*}}
\newcommand{\ov}{\overline}
\newcommand{\mb}{\mathbb}
\newcommand{\mc}{\mathcal}
\newcommand{\mf}{\mathfrak}

\newcommand{\ev}{{\rm ev}}

\title{Counting Pointlike Instantons Virtually without Gluing}

\begin{document}

\author[Tian]{Gang Tian}
\address{
BICMR and SMS\\
Beijing University\\
Beijing, P.R.China, 100871}
\email{gtian@math.pku.edu.cn}

\author[Xu]{Guangbo Xu}
\address{
Department of Mathematics \\
Texas A{\&}M University\\
College Station, TX 77843 USA
}
\email{guangboxu@math.tamu.edu}

\date{\today}

\maketitle

\begin{abstract}
We use the first author's idea on ``constructing virtual fundamental cycles without gluing'' to define the quantum correction term between the gauged linear sigma model and the nonlinear sigma model. 
\end{abstract}

\setcounter{tocdepth}{1}
\tableofcontents

\section{Introduction}

This paper aims at constructing an invariant appeared in Witten's gauged linear sigma model (GLSM) using a novel idea in the theory of virtual fundamental cycle (VFC) which was originally due to the first author (see \cite{Tian_no_glue} and the thesis of Shaofeng Wang \cite{Shaofeng_thesis}). The invariant, denoted by ${\mf c}$, which is formally a weighted count of pointlike instantons, plays an important role in GLSM and mirror symmetry. The ``no-gluing'' approach of the VFC construction greatly simplified the technicality involved in the definition.

\subsection{Counting pointlike instantons in gauged linear sigma model}

Witten's GLSM \cite{Witten_LGCY} is a powerful framework which connects two basic superconformal field theories: the nonlinear sigma model and the orbifold Landau--Ginzburg model. It has been regarded as an ideal setting for proving the Landau--Ginzburg/Calabi--Yau correspondence \cite{Martinec}\cite{GVW}\cite{VafaW}. It also plays a crucial role in Hori--Vafa's beautiful ``proof'' of mirror symmetry \cite{Hori_Vafa}. In 2012 the authors initiated the program of establishing the mathematical foundation of GLSM using methods from symplectic geometry. After \cite{Tian_Xu, Tian_Xu_2, Tian_Xu_3, Tian_Xu_4, Tian_Xu_geometric}, we have defined a cohomological field theory for GLSM spaces satisfying the ``geometric phase'' condition. Meanwhile algebraic geometers have developed parallel theories, see \cite{FJR_GLSM}\cite{CLLL_2019}\cite{CFGKS} etc.

We give a brief review of our prior results. Let $(V, G, W, \mu)$ be a GLSM space, which roughly means\footnote{For the precise definition, see Section \ref{section2}.}: $V$ is a K\"ahler manifold, $G$ is a complex reductive Lie group (whose maximal compact subgroup is $K$) acting on $V$, $W: V \to {\mb C}$ is a $G$-invariant holomorphic function, and $\mu: V \to {\mf k}^*$ is a moment map for the restricted $K$-action. The {\bf classical vacuum} is 
\beqn
X = ({\rm Crit} W \cap \mu^{-1}(0))/K.
\eeqn
Under the {\bf geometric phase} condition and some other technical assumptions, the classical vacuum is a compact K\"ahler orbifold. For example, all hypersurfaces in projective spaces can be realized as the classical vacua of certain GLSM spaces satisfying the geometric phase condition. On the other hand, Landau--Ginzburg phases or hybrid phases are not under consideration here. 

In \cite{Tian_Xu_geometric} the authors constructed the GLSM correlation function under the geometric phase condition. The correlation functions are multilinear maps
\beqn
\langle \cdots \rangle_{g,n}^{\rm GLSM}: H^*(X; \Lambda)^{\otimes n} \to H_*( \ov{\mc M}_{g,n}; \Lambda).
\eeqn
Here $\Lambda$ is a certain Novikov ring (with rational coefficients), $g$ is the genus, $n$ is the number of marked points, such that $2g + n \geq 3$. The construction uses a detailed version of the VFC construction introduced in \cite{Li_Tian} (with gluing) where transversality is treated in the topological ($C^0$) sense. Our GLSM correlation functions satisfy axioms of cohomological field theories (CohFT) \cite{Manin_2}; they can be regarded as the symplectic counterpart of the invariants constructed via algebraic methods (see \cite{CCK_quasimap}\cite{FJR_GLSM}\cite{CFGKS}\cite{Favero_Kim_2020}). 

The relation between GLSM invariants and Gromov--Witten invariants of the classical vacuum is an important question. In the special case when $W = 0$, Dietmar Salamon conjectured that such a relation can be interpreted as a ``quantum deformation'' of the classical Kirwan map. Using the adiabatic limit method Gaio and Salamon \cite{Gaio_Salamon_2005} verified a special case of this relation. The quantum Kirwan map conjecture has been pursued by Ziltener \cite{Ziltener_book} and solved in the algebraic case by Woodward \cite{Woodward_15}. In the more interesting case when $W \neq 0$, we proposed the following conjecture in \cite{Tian_Xu_2017, Tian_Xu_geometric}.

\begin{conj}\label{conj1}
Let $(V, G, W, \mu)$ be a GLSM space satisfying the geometric phase condition with classical vacuum $X$ being a compact K\"ahler manifold. Then there is a class ${\mf c} \in H^*(X; \Lambda^+)$ such that for all $(g, n)$ with $2g + n \geq 3$, one has 
\beqn
\langle \alpha_1, \ldots, \alpha_n \rangle_{g, n}^{\rm GLSM} = \sum_{k=0}^\infty \frac{1}{k!} \langle \alpha_1, \ldots, \alpha_n, \underbrace{{\mf c}, \ldots, {\mf c}}_{k} \rangle_{g, n+k}^{\rm GW}.
\eeqn
\end{conj}

We remark that, the above formula does not only connect two theories on the A-side, but is also closely related to mirror symmetry for Calabi--Yau manifolds. Indeed, as explained by physicists \cite{Hori_Vafa}, the mirror map between the symplectic moduli and the K\"ahler moduli for certain mirror pairs is essentially equivalent to the class ${\mf c}$. Recently a mathematical verification of this prediction has been given in \cite{CK_2020} using quasimaps. We will discuss the relation between ${\mf c}$ and the mirror map in our setting in future works.

In this paper, we establish the first step towards Conjecture \ref{conj1}. As the quantum Kirwan map is defined by counting a kind of object called ``affine vortices,'' the class ${\mf c}$ is defined by counting a generalization called {\bf pointlike instantons}. 

\begin{thm}
Let $(V, G, W, \mu)$ be a gauged linear sigma model space in geometric phase. Suppose the classical vacuum is a manifold. Then the virtual counts of pointlike instantons define a cohomology class 
\beqn
{\mf c} \in H^*(X; \Lambda_+)
\eeqn
which is independent of various choices. (Here $\Lambda_+$ is the Novikov ring with positive valuation.) In particular, when the classical vacuum is a Calabi--Yau manifold, ${\mf c}$ is homogeneous of degree $2$. 
\end{thm}

The definition of the correction term is based on the virtual fundamental cycle construction. In our forthcoming paper, we will prove Conjecture \ref{conj1} using the adiabatic limit method which generalizes the one used in \cite{Gaio_Salamon_2005}.

\subsection{VFC without gluing}
 
We introduce the first author's idea which says that gluing is unnecessary for certain VFC constructions. In symplectic geometry, to define Gromov--Witten invariants or Floer homology of a general symplectic manifold, one needs to regularize the moduli space of stable pseudoholomorphic maps. As one cannot achieve transversality by perturbing global geometric data (such as the almost complex structure),  mathematicians invented the virtual technique, which allows one to perturb the moduli space locally but abstractly. Basically, after finite-dimensional reduction (as did in \cite{Li_Tian} and \cite{Fukaya_Ono}), a moduli space $\ov{\mc M}$ locally is characterized by a Kuranishi chart
\beqn
C = (U, E, S, \psi, F)
\eeqn
where $U$ is a smooth or topological orbifold, $E \to U$ is an orbifold vector bundle, $S : U \to E$ is a section, $\psi: S^{-1}(0) \to \ov{\mc M}$ is a continuous map into the moduli space which is a homeomorphism onto an open subset $F \subset \ov{\mc M}$. In general the section $S$ (the Kuranishi map) is not transverse, but one can perturb it so that the singular locus $S^{-1}(0)$ is replaced by a regular (possibly branched) one. By covering the moduli space with finitely many Kuranishi charts and construct compatible perturbations inductively, the perturbed vanishing loci patched together to give a compact cycle, whose cobordism class is independent of the perturbation. 

One can see that the orbifold structure of the Kuranishi chart is essential for the notion of transversality and hence the whole VFC construction. The standard construction of such orbifold structure near a stable map with a nodal domain depends on the ``gluing construction.'' Namely, one has to show that the gluing parameters, which parametrize the resolutions of the nodal points, are part of the coordinates on the chart. This construction is often highly involved. 

The main observation of the first author is that to define a good perturbation, one only needs a weaker ``stratified'' structure on charts. Notice that the moduli space and very local chart are naturally stratified by the combinatorial types of curves. If each stratum of a chart is an orbifold, then a $C^0$ perturbation can be constructed inductively so that it is transverse over each stratum. As dimensions differ at least by two between strata, such a perturbation gives a pseudocycle in the $C^0$ category. On the other hand, this approach works more conveniently with the VFC theory in the $C^0$-category, which was emphasized in Li--Tian's original work \cite{Li_Tian} and was carried out in detail in the recent work \cite{Tian_Xu_geometric} in the setting of GLSM.



One can also see that the ``VFC without gluing'' approach is particularly convenient in the application considered in this paper. First, as singular configurations in the compactified moduli space consist of objects of different types (holomorphic curves and pointlike instantons), the gluing construction would be much more complicated than gluing holomorphic curves. Second, the space of gluing parameters is not a vector space but an affine variety. The no-gluing approach saves us the effort of describing such more general charts. 

We remark that the ``VFC without gluing'' method cannot save one's effort of codimension one gluing. For example, to prove results such as $d^2 = 0$ in Floer theory, one still needs to gluing broken trajectories. It is also the case in the forthcoming paper where we need to perform a codimension one gluing to prove Conjecture \ref{conj1}.

\subsection{Organization}

This paper is organized as follows. In Section \ref{section2} we review the geometric setting of GLSM and discuss basic facts about pointlike instantons. In Section \ref{section3} we describe the compactification of the moduli spaces of pointlike instantons. In Section \ref{section4} we set up some abstract framework of our construction. In Section \ref{section5} we give the details of the construction of good coordinate systems and the definition of the virtual count. In Appendix \ref{appendixa} we prove some technical results about pointlike instantons.

\subsection{Acknowledgement}

G.X. would like to thank Professor Kenji Fukaya for stimulating discussions.

\section{Gauged Linear Sigma Model and Pointlike Instantons}\label{section2}

In this section we give the basic geometric setting and review properties of pointlike instantons.

\begin{defn}
A {\bf gauged linear sigma model space} (GLSM space for short) is a quadruple $(V, G, W, \mu)$ where $V$ is a noncompact K\"ahler manifold equipped with a holomorphic ${\mb C}^*$-action (the $R$-symmetry), $G$ is a reductive complex Lie group with maximal compact subgroup $K$ acting on $V$ (which commutes with the $R$-symmetry), $W: V \to {\mb C}$ is a $G$-invariant holomorphic section which is homogeneous of degree $r\geq 1$ with respect to the $R$-symmetry, and $\mu: V \to {\mf k}^*$ is a moment map of the $K$-action with $0 \in {\mf k}^*$ a regular value of $\mu$. 
\end{defn}

We define the {\it semi-stable locus} of $(V, G, W, \mu)$ to be 
\beqn
V^{\rm ss}:= G \mu^{-1}(0).
\eeqn
This is an open subset of $V$. 

\begin{defn}
$(V, G, W, \mu)$ is said to satisfy the {\bf geometric phase condition} if the restriction of $W$ to $V^{\rm ss}$ is a holomorphic Morse--Bott function. In this situation, the {\bf classical vacuum} of $(V, G, W, \mu)$ is the quotient 
\beqn
X:= ({\rm Crit} W \cap \mu^{-1}(0))/ K \subset \mu^{-1}(0)/K.
\eeqn
\end{defn}

We also need three additional hypotheses. The first one is important to the proof of compactness. 

\begin{hyp}
There exists $\xi_W\in {\rm Center}({\mf k})$ and a continuous function $\tau \mapsto c_W(\tau)$ for $\tau\in {\rm Center}({\mf k})$ satisfying the following condition. If we define ${\mc F}_W:= \mu \cdot \xi_W$, then ${\mc F}_W$ is proper and 
\beqn
\begin{array}{c}
x\in {\rm crit} W \cap V^{\rm ss},\ \xi \in T_x V\\
{\mc F}_W(x) \geq c_W(\tau)
\end{array} \Longrightarrow \left\{ \begin{array}{c}  \langle \nabla_\xi \nabla {\mc F}_W(x), \xi \rangle + \langle \nabla_{J\xi} \nabla {\mc F}_W(x), J \xi \rangle \geq 0,\\
                      \langle \nabla {\mc F}_W(x), J {\mc X}_{\mu(x) - \tau}(x) \rangle \geq 0.   \end{array} \right.
\eeqn
\end{hyp}

The second additional hypothesis, which is not completely necessary, is to simplify the discussion if one is not interested in the more general orbifold theory.
\begin{hyp}
The $K$-action on $\mu^{-1}(0)$ is free.
\end{hyp}

It follows that the classical vacuum is a compact K\"ahler manifold. Lastly, we assume that $V$ is aspherical in order to rule out bubbling in $V$.

\begin{hyp}
For all smooth maps $u: S^2 \to V$ there holds $\int_{S^2} u^* \omega = 0$. 
\end{hyp}

\begin{rem}
A special situation is when the superpotential $W$ vanishes. In this situation we can make other assumptions more general. Indeed, when $W = 0$, it is possible that $V$ is compact. Then we only need to assume that $V$ is symplectic with a Hamiltonian $K$-action. This is the setting of the construction of gauged Gromov--Witten invariants \cite{Mundet_thesis, Mundet_2003, Cieliebak_Gaio_Salamon_2000, Cieliebak_Gaio_Mundet_Salamon_2002, Mundet_Tian_2009}.

\end{rem}

\subsection{Pointlike instantons}

We first introduced the gauged Witten equation in \cite{Tian_Xu}. The version under the current setting was introduced in \cite{Tian_Xu_geometric}. In this paper we only consider the case when the domain curve is the complex plane. 

A {\bf gauged map} from ${\mb C}$ to $V$ is a triple of maps 
\beqn
(u, \phi, \psi): {\mb C} \to V \times {\mf k}\times {\mf k}.
\eeqn
We should regard this triple as the presentation of a connection $A:= d + \phi ds + \psi dt$ of the trivial $K$-bundle $P \to {\mb C}$ and a section of the associated bundle $P\times_K V$ written with respect to the trivialization. The gauged Witten equation reads
\beq\label{eqn21}
\left\{ \begin{array}{rcc} \partial_s u + {\mc X}_\phi(u) + J(\partial_t u + {\mc X}_\psi(u)) + \nabla W(u) & = & 0,\\
                          \partial_s \psi - \partial_t \phi + [\phi, \psi] + \mu(u) & = & 0.
\end{array} \right.
\eeq
The energy of a gauged map is defined by 
\beqn
E(u, \phi, \psi) = \frac{1}{2} \left( \| d_A  u \|_{L^2}^2 + \| F_A \|_{L^2}^2 + \| \mu(u) \|_{L^2}^2 \right) + \| \nabla W(u) \|_{L^2}^2.
\eeqn

We always consider only {\bf bounded solutions}, i.e., those solutions with $\ov{u({\mb C})}$ compact and with $E(u, \phi,\psi)$ finite. 

\begin{rem}
To write down the gauged Witten equation over a general Riemann surface, one needs an extra structure called the {\bf $r$-spin structure} if the degree of $W$ is $r$. In the situation of the current paper the $r$-spin structure over the complex plane is always trivial so we do not see it explicitly.
\end{rem}

There have been many studies of pointlike instantons in the special situation when the superpotential vanishes. When $W = 0$, a solution to \eqref{eqn21} is called an {\bf affine vortex}, i.e., a solution $(u, \phi, \psi)$ to
\beq\label{eqn22}
\left\{ \begin{array}{rcc} \partial_s u + {\mc X}_\phi(u) + J( \partial_t u + {\mc X}_\psi(u)) & = & 0,\\
\partial_s \psi - \partial_t \phi + [\phi, \psi] + \mu(u) & = & 0.
\end{array}\right.
\eeq
These objects first appeared in the Ginzburg--Landau theory of superconductivity (see \cite{Jaffe_Taubes}). In symplectic geometry, Gaio--Salamon realized their importance in the ``quantum Kirwan map'' conjecture (see \cite{Gaio_Salamon_2005}). Ziltener \cite{Ziltener_Decay, Ziltener_thesis, Ziltener_book} has made significant advancement towards this conjecture by establishing basic bubbling and Fredholm theories for affine vortices. When $V$ is a projective manifold or a vector space, Venugopalan--Woodward classifies all affine vortices \cite{VW_affine}, which was used in Woodward's proof of the quantum Kirwan map conjecture \cite{Woodward_15} in this setting. Wang and Xu \cite{Wang_Xu} established basic analytical results of affine vortices with Lagrangian boundary conditions. Venugopalan and Xu \cite{Venugopalan_Xu} constructed local Fredholm model for the moduli spaces.

An important result we will need is that all pointlike instantons are also affine vortices. 

\begin{prop}\label{prop28}
If $(u, \phi, \psi)$ is a bounded solution to the gauged Witten equation, then $\nabla W(u) \equiv 0$. Therefore, every pointlike instanton is an affine vortex in $V$. 
\end{prop}

\begin{proof}
See the appendix.
\end{proof}

The above proposition allows us to use existing results of affine vortices from \cite{Gaio_Salamon_2005} \cite{Ziltener_Decay} \cite{Ziltener_book} \cite{Venugopalan_Xu}. In particular one has the following ``removable singularity'' result.

\begin{cor}\cite{Gaio_Salamon_2005}\cite{Ziltener_book}
For each pointlike instanton $(u, \phi,\psi)$, there exists a smooth $K$-bundle $P \to S^2\cong {\mb C} \cup \{\infty\}$ and a continuous section $\tilde u: S^2 \to P(V)$ such that with respect to certain trivializations of $\tilde P|_{{\mb C}}$, one has $\tilde u|_{{\mb C}} = u$. 
\end{cor}

It follows that each pointlike instanton $u$ has a {\bf degree}
\beqn
d:= [u]:= \tilde u_*[S^2] \in H_2^K(V; {\mb Z}).
\eeqn
By the energy identity of vortices (see \cite{Cieliebak_Gaio_Mundet_Salamon_2002}), if $(u, \phi, \psi)$ has degree $d$, then
\beqn
E(u) = \langle \omega^K, d \rangle\in [0, +\infty)
\eeqn
where $\omega^K\in H^2_K(V; {\mb R})$ is the equivariant cohomology class represented by the equivariant form $\omega - \mu$. 

\subsection{Moduli spaces}

A gauge transformation, in most of the cases considered in this paper, is a smooth map $g: {\mb C} \to K$. If $u = (u, \phi, \psi)$ is a gauged map, then $u_g= (u_g, \phi_g, \psi_g):= g\cdot u$ is the gauged map defined by
\beqn
u_g(z) = g(z) u(z),\ \phi_g(z) = {\rm Ad}_{g(z)} \phi(z) + g(z)^{-1} \partial_s g(z),\ \psi_g(z) = {\rm Ad}_{g(z)} \psi(z) + g(z)^{-1} \partial_t g(z).
\eeqn
The gauged Witten equation \eqref{eqn21} is gauge invariant. If $u$ is a solution and $g$ is a gauge transformation, then $u_g$ is also a solution.

Let $\tilde {\mc M}(V)$ denote the set of gauge equivalence classes of solutions to \eqref{eqn21}. This set is equipped with the $C^\infty$-topology which implies the following notion of sequential convergence: a sequence of gauge equivalence classes $[u_i] \in \tilde {\mc M}(V)$ converges to $[u_\infty]$ if there exists a sequence of gauge transformations $g_i: {\mb C} \to K$ such that $u_{i, g_i}$ converges in $C^\infty$ to $u_\infty$ over any compact subset of ${\mb C}$. This topology is clearly Hausdorff. We also denote by $\tilde {\mc M}(V,d) \subset \tilde {\mc M}(V)$ the subset of gauge equivalence classes of solutions with a specified degree $d$. 

There is also a translation symmetry which is very important for a finite count. There is a ${\mb C}$-action on $\tilde {\mc M}(V)$ which is free on each $\tilde {\mc M}(V, d)$ for all nonzero degree $d$. Denote 
\beqn
{\mc M}(V, d):= \tilde {\mc M}(V, d) / {\mb C}
\eeqn
which is equipped with the quotient topology.

\subsection{Adiabatic limit}

In the compactification of the moduli space one needs to discuss the adiabatic limit of the gauged Witten equation or the vortex equation. Consider the vortex equation over a large disk $B_R \subset {\mb C}$ with a possibly nonflat area form $\sigma(s,t) ds dt$. If we identify $B_R$ with $B_1$ via the rescaling, then the gauged Witten equation over $B_R$ becomes the following equation over $B_1$:
\beq\label{eqn23}
\left\{ \begin{array}{rcc} \partial_s u + {\mc X}_\phi(u) + J( \partial_t u + {\mc X}_\psi) & = & 0,\\
\partial_s \psi - \partial_t \phi + [\phi, \psi] + R^2 \sigma(s,t) \mu(u) & = & 0. \end{array}\right.
\eeq
When $R \to \infty$ one can see naively $\mu(u) \to 0$, formally implying that solutions converge to holomorphic curves in the classical vacuum. As we only need to consider in this paper the situation when the image of the map $u$ is sufficiently close to the level set $\mu^{-1}(0)$, we can regard the solution as a solution to the vortex equation in the smooth part of $dW^{-1}(0)$, thanks to Proposition \ref{prop28}. What we need is a few estimates most of which were derived in \cite{Gaio_Salamon_2005}.

We first define the notion of convergence in adiabatic limit. Recall that there is a holomorphic map 
\beqn
\pi_V: V^{\rm ss} \to X
\eeqn
which sends $x \in V^{\rm ss}$ to its $G$-orbit. 

\begin{defn}\label{defn210}
Let $\Omega \subset {\mb C}$ be an open subset and let $\Omega_i \subset \Omega$ be an exhaustive open subset of $\Omega$, i.e., any compact subset of $\Omega$ is contained in $\Omega_i$ for sufficiently large $i$. Let $\sigma_i: \Omega_i \to (0, +\infty)$ be a sequence of smooth functions converging in $C^\infty_{\rm loc}$ to $\sigma_\infty: \Omega \to (0, +\infty)$. Let $R_i \to + \infty$ be a sequence of real numbers and let $(u_i, \phi_i, \psi_i)$ be a sequence of solutions to \eqref{eqn23} for $R = R_i$ and $\sigma = \sigma_i$ over $\Omega_i$. We say that $u_i$ {\bf converges in the sense of adiabatic limit} to a holomorphic map $u_\infty: \Omega \to X$ if the following are satisfied.
\begin{enumerate}
\item For all compact subset $Z \subset \Omega$ there holds
\beqn
\sup_i \left( \| u_{i,s} \|_{L^\infty(Z)} + R_i \| \sqrt{\sigma_i} \mu(u_i)\|_{L^\infty(Z)}\right) < \infty.
\eeqn
Here we use the notation 
\begin{align*}
&\ u_{i,s} = \partial_s u_i + {\mc X}_{\phi_i}(u_i),\ &\ u_{i,t} = \partial_t u_i + {\mc X}_{\psi_i}(u_i).
\end{align*}
It follows that for $i$ sufficiently large, $u_i(Z)$ is contained in $V^{\rm ss}$. 

\item The composition $\pi_V \circ u_i$ converges to the holomorphic map $u_\infty$ over all compact subsets of $\Omega$. 
\end{enumerate}
\end{defn}

In the construction of local charts of the moduli spaces, one needs to use additional codimension two submanifolds which are transverse to vortices or holomorphic curves. The following result implies that for particular choices of such submanifolds, transversality persists in the adiabatic limit. 

\begin{lemma}\label{lemma211}
Suppose $u_i$ converges to $u_\infty$ as defined above. Then for all compact subsets $Z \subset \Omega$ there holds
\beq\label{eqn24}
\lim_{i \to \infty} \Big( \| d\mu(u_i)(u_{i,s})\|_{L^\infty(Z)} + \| d\mu(u_i)(J u_{i,s}) \|_{L^\infty(Z)} \Big)  = 0.
\eeq
\end{lemma}

\begin{proof}
The limiting holomorphic curve $u_\infty$ can be lifted to a gauged map $(u_\infty, \phi_\infty, \psi_\infty)$ solving the equation
\begin{align*}
&\ \partial_s u_\infty + {\mc X}_{\phi_\infty}(u_\infty) + J( \partial_t u_\infty + {\mc X}_{\psi_\infty}(u_\infty)) = 0,\ & \ \mu(u_\infty) \equiv 0.
\end{align*}
It follows that 
\beq\label{eqn25}
d\mu(u_\infty)(u_{\infty, s}) = d\mu(u_\infty)(J u_{\infty, s}) \equiv 0.
\eeq
By \cite[Proposition 38]{Ziltener_book}, after applying a suitable sequence of gauge transformations, for a subsequence $i_k$, $u_{i_k}$ converges to $u_\infty$ in $C^1_{\rm loc}$ and $\phi_{i_k} ds + \psi_{i_k} dt$ converges to $\phi_\infty ds + \psi_\infty dt$ in $C^0_{\rm loc}$.  Then it follows that the covariant derivative $u_{i_k, s}$ converges to $u_{\infty, s}$ uniformly on compact subsets. Then by \eqref{eqn25}, one has 
\beqn
\lim_{i_k \to \infty} \Big( \| d\mu(u_{i_k})(u_{i_k,s})\|_{L^\infty(Z)} + \| d\mu(u_{i_k})(J u_{i_k,s}) \|_{L^\infty(Z)} \Big)  = 0.
\eeqn
As any subsequence of $u_i$ contains a subsequence satisfying the above, \eqref{eqn24} follows. 
\end{proof}

\section{Ziltener Compactification}\label{section3}

In this section we describe the natural compactification of the moduli space of pointlike instantons. Our compactification can be viewed as a special case of the one of the moduli space of affine vortices given by Ziltener \cite{Ziltener_book}.

\subsection{Trees} 

We first recall notions about trees. A tree $\Gamma$ consists of a nonempty set of vertices $V_\Gamma$, a set of edges $E_\Gamma$, and a set of tails (or semi-infinite edges) $T_\Gamma$. In this paper we only consider {\it rooted trees}, i.e., a distinguished vertex $v_\infty \in V_\Gamma$ is specified. Then $V_\Gamma$ can be given a canonical partial order $\leq$ such that $v_\infty$ is the only minimal element. We denote by $v \succ v'$ if $v$ and $v'$ are adjacent and $v'$ is closer to $v_\infty$ and denote by $v \overset{e}{\succ} v'$ if in addition the edge connecting them is $e$. In addition to roots, we also endow each tree with a {\it scaling}, i.e., a function 
\beqn
{\mf s}: V_\Gamma \to \{0, 1, \infty\}.
\eeqn
We require that ${\mf s}$ is order-reversing with respect to the ordering $0 \leq 1 \leq \infty$. Moreover, we require 
\begin{itemize}
\item For any path $v_1 \succ \cdots \succ v_k$ in $\Gamma$ if ${\mf s}(v_1) \leq 1$ and ${\mf s}(v_k) \geq 1$, then there exists a unique $v_i$ in this path such that ${\mf s}(v_i) = 1$. 
\end{itemize}
In other words, vertices of scale $1$ are not adjacent to each other and any two vertices of scale $0$ and $\infty$ respectively are never adjacent. 

\begin{notation}
For each edge $e\in E_\Gamma$, denoted by $v(e)\in V_\Gamma$ one of the two vertices connected to $e$ which is closer to the root. For each tail $t \in T_\Gamma$, denote by $v(t)\in V_\Gamma$ the vertex to which $t$ is attached. 
\end{notation}

\begin{convention}
In this paper, all trees $\Gamma$ we are considering satisfy $V_\Gamma^0 = \emptyset$ and that all tails are attached to vertices in $V_\Gamma^1$. 
\end{convention}

A scaled tree $\Gamma$ (under the above convention) is {\it stable} if each $v \in V_\Gamma^1$ has at least two tails attached and each $v\in V_\Gamma^\infty$ has valence at least three. 

\subsection{Scaled curves}

\begin{defn}
Let $\Gamma$ be a scaled tree. A scaled curve of combinatorial type $\Gamma$ is a collection 
\beqn
{\mc C} = ( (\Sigma_v, \varphi_v)_{v\in {\mb C}},\ (w_e)_{e \in E_\Gamma},\ (z_t)_{t\in T_\Gamma}).
\eeqn
Here each $\Sigma_v$ is a Riemann surface and $\varphi_v: \Sigma_v \to {\mb C}$ is a biholomorphism, each $w_e$ is a point on $\Sigma_{v(e)}$, and each $z_t$ is a point on $\Sigma_{v(t)}$. This collection needs to satisfy the following condition.
\begin{itemize}
\item For each $v \in V_\Gamma$, the collection of points
\beqn
{\bm w}_v:= (w_e)_{v(e) = v} \cup (z_t)_{v(t) = v}
\eeqn
are distinct. They are called {\it special points} on $\Sigma_v$.
\end{itemize}
\end{defn}

Now we define the notion of isomorphism between scaled curves. 

\begin{defn}\label{defn34}
For $i = 1, 2$, let 
\beqn
{\mc C}^i = ((\Sigma_v^i, \varphi_v^i)_{v\in V_\Gamma},\ (w_e^i)_{e\in E_\Gamma},\ (z_t^i)_{t\in T_\Gamma}) 
\eeqn
be two scaled curves. An {\bf isomorphism} from ${\mc C}^1$ to ${\mc C}^2$ is a collection $(\phi_v)_{v\in V_\Gamma}$ of biholomorphisms $\phi_v: \Sigma_v^1 \to \Sigma_v^2$ satisfying the following conditions.
\begin{enumerate}
\item For each $e \in E_\Gamma$, $\phi_{v(e)} (w_e^1) = w_e^2$.

\item For each $t \in T_\Gamma$, $\phi_{v(t)} (z_t^1) = z_t^2$.

\item For each $v \in V_\Gamma^1$, the composition
\beqn
\varphi_v^2 \circ \phi_v \circ (\varphi_v^1)^{-1}: {\mb C} \to {\mb C}
\eeqn
is a translation.
\end{enumerate}
\end{defn}

We say that a scaled curve of type $\Gamma$ is stable if $\Gamma$ is stable. It is easy to see that a scaled curve is stable if and only if it has a trivial automorphism group.

\subsection{Stable pointlike instantons}

\begin{defn}
Let $\Gamma$ be a tree. A {\bf singular pointlike instanton} of type $\Gamma$ is a pair
\beqn
({\mc C}, {\bm u}) = ({\mc C}, (u_v)_{v\in V_\Gamma})
\eeqn
where ${\mc C}$ is a scaled curve of type $\Gamma$ and 
\begin{enumerate}
\item for each $v \in V_\Gamma^\infty$, $u_v: {\mb C}\to X$ is a holomorphic sphere in $X$;

\item for each $v \in V_\Gamma^1$, $u_v = (u_v, \phi_v, \psi_v)$ is a pointlike instanton.
\end{enumerate}
They satisfy
\begin{itemize}

\item ({\bf Matching condition}) For each edge $v \overset{e}{\succ} v'$ of $\Gamma$, there holds
\beqn
\ev_\infty(u_v) = u_{v'}(w_e) \in X.
\eeqn

\noindent The singular pointlike instanton is called a {\bf stable} pointlike instanton if it satisfies

\item ({\bf Stability condition}) If $v\in V_\Gamma$ is unstable, then the energy of $u_v$ is positive. 
\end{itemize}
\end{defn}

We can also define the degree of a singular pointlike instanton. The degree of the holomorphic sphere $u_v: S^2 \to X$, which is a homology class in the classical vacuum, can be lifted to an equivariant homology class via the composition of the map
\beqn
H_2(X; {\mb Z}) \to H_2(V \qu G; {\mb Z}) \to H_2^K(V; {\mb Z}).
\eeqn
We still call the corresponding equivariant homology class the {\it degree} of the holomorphic sphere. Then we define the degree of a stable pointlike instanton to be the sum of the degrees of all of its component, which is an equivariant homology class in $V$ with integral coefficients. 

We define the combinatorial type of a singular pointlike instanton to be a tree $\Gamma$ together with a list of degrees $(d_v)_{v \in V_\Gamma}$ satisfying the stability condition: if $v\in V_\Gamma$ is unstable, then $d_v \neq 0$. Such combinatorial data are called {\bf map types} and denoted by $\lambda$.

\begin{defn}
Let $({\mc C}^i, {\bm u}^i)$ ($i = 1, 2$) be two singular pointlike instantons of domain type $\Gamma$. An isomorphism from $({\mc C}^1, {\bm u}^1)$ to $({\mc C}^2, {\bm u}^2)$ is a pair 
\beqn
((\phi_v)_{v\in V_\Gamma}, (g_v)_{v \in V_\Gamma^1})
\eeqn
where $(\phi_v)_{v\in V_\Gamma}$ is an isomorphism from ${\mc C}^1$ to ${\mc C}^2$ (see Definition \ref{defn34}) and $g_v: \Sigma_v^1 \to K$ is a gauge transformation for all $v \in V_\Gamma^1$ such that 
\begin{enumerate}

\item For all $v \in V_\Gamma^1$, $(u_v^2 \circ \phi_v)_{g_v} = u_v^1$ as gauged maps.

\item For all $v \in V_\Gamma^\infty$, $u_v^2 \circ \phi_v = u_v^1$ as maps into $X$. 
\end{enumerate}
\end{defn}

It is routine to check that a singular pointlike instanton is stable if and only if its automorphism group is finite. Then for each map type $\lambda$, denote by 
\beqn
{\mc M}(V)_\lambda
\eeqn
the set of isomorphism classes of singular pointlike instantons of type $\Gamma$. Then for each degree $d \in H_2^K( V; {\mb Z})$, define
\beqn
\ov{\mc M}(V;d):= \bigsqcup_{d_\lambda = d} {\mc M}(V)_\lambda.
\eeqn

\subsection{Topology on the moduli space}

It is time to define the topology of the moduli space $\ov{\mc M}(V;d)$ via a notion of sequential convergence. We need first to define how a sequence of smooth objects converge to a singular one. 

In the following definition, since the domains of components are all regarded as ${\mb C}$, M\"obius transformations which are typically used in defining convergence of stable holomorphic curves are replaced by affine linear maps $z \mapsto \epsilon z + w$; when $\epsilon = 1$ we call them translations.

\begin{defn}\label{defn37}
Let $u_1, \ldots, $ be a sequence of (smooth) pointlike instantons. Let $({\mc C}, {\bm u})$ be a stable pointlike instanton of map type $\lambda$ and domain type $\Gamma$. We say that $\{ u_i \}$ converges (modulo translation and gauge transformation) to $({\mc C}, {\bm u})$ if the following conditions hold.

\begin{enumerate}

\item For each $v \in V_\Gamma^1$, there exists a sequence of translations $\varphi_i: \Sigma_v \to {\mb C}$ and a sequence of gauge transformations $g_i: \Sigma_v \to K$ such that $(u_i \circ \varphi_i)_{g_i}$ converges to $u_v = (u_v, \phi_v, \psi_v)$ away from the set $W_v$. 

\item For each $v \in V_\Gamma^\infty$, there exists a sequence of biholomorphic maps $\varphi_{i, v}: \Sigma_v \to {\mb C}$ such that the sequence of pullback area forms $\varphi_{i, v}^* (ds dt)$ converge to infinity on compact subsets of $\Sigma_v$ and the sequence $u_i \circ \varphi_i$ of gauged maps on $\Sigma_v$ converges to the holomorphic map $u_v: \Sigma_v \to X$ (in the adiabatic sense, see Definition \ref{defn210}). 

\item For each edge $v  \overset{e}{\succ} v'$, there holds $\varphi_{i,v'}^{-1}\circ \varphi_{i,v}$ converges uniformly on compact subsets of ${\mb C}$ to $w_e \in \Sigma_{v'}$.

\item There is no energy lost, i.e., 
\beqn
\lim_{i \to \infty} E(u_i) = E({\mc C}):= \sum_{v_\alpha\in V_\Gamma} E(u_\alpha).
\eeqn
\end{enumerate}
\end{defn}

The following compactness theorem follows directly from Ziltener's compactness theorem for affine vortices \cite[Theorem 3]{Ziltener_book} and Proposition \ref{prop28}.

\begin{thm}\cite[Theorem 3]{Ziltener_book} Let $u_i$ be a sequence of pointlike instantons satisfying $\sup_i E(u_i) < \infty$. Then there exists a subsequence (still indexed by $i$) which converges modulo translation and gauge transformation to a stable pointlike instanton. 
\end{thm}

\begin{proof}
First, by the same method of \cite[Theorem 5.11]{Tian_Xu_geometric}, one can prove that the image of $u_i$ are contained in a common compact subset of $V$. Further, by Proposition \ref{prop28}, $u_i$ are affine vortices in $V$. Then by Ziltener's compactness theorem, a subsequence converges to a stable affine vortex. Since all $u_i$ are contained in $dW^{-1}(0)$, so is the limit. Hence the limiting object is also a stable pointlike instanton. 
\end{proof}

From Definition \ref{defn37}, combining with the notion of Gromov convergence of stable pseudoholomorphic curves, one can write down a long and tedious definition of the notion of sequential convergence of stable pointlike instantons. Using the same argument as \cite[Section 5.6]{McDuff_Salamon_2004}, one can show that this topology is compact, Hausdorff, and first-countable, whose set of convergence sequences is the same as the set of Gromov convergent sequences. We leave the details to the reader. We do not prove that the topology is second countable and hence metrizable, as it is already very complicated for pseudoholomorphic maps (see \cite{McDuff_Salamon_erratum_2}). What we need is that every stratum of the moduli space is metrizable (see Lemma \ref{lemma56} and Remark \ref{rem57}).

\section{The Abstract Setup}\label{section4}

Our purpose is to define a homology class in the classical vacuum (with rational coefficients) which is formally the pushforward of the fundamental class (VFC) $[\ov{\mc M}(V;d)]^{\rm vir}$ under the evaluation map 
\beqn
\ev_\infty: \ov{\mc M}(V;d) \to X.
\eeqn
The traditional construction of the VFC requires a sophisticated gluing construction which would be a generalization of the gluing result of the author \cite{Xu_glue}. 

The basic idea of ``VFC without gluing'' is to construct a pseudocycle inside a virtual neighborhood of $\ov{\mc M}(V;d)$ without worrying about the structure between strata. Here we further simplify the construction by considering only index zero moduli spaces. Indeed, suppose the expected dimension of $\ov{\mc M}(V;d)$ is $m_d - 2$ (see Proposition \ref{propa4}), where $m_d$ is the index of the linearized Fredholm operator and $-2$ comes from the translation symmetry. Assume that we can choose a set of smooth cycles
\beqn
S_1, \ldots, S_k \subset X
\eeqn
whose Poincar\'e duals form a basis of $H^{m_d-2}(X; {\mb Q})$. Let 
\beqn
\ov{\mc M}_{S_i}(V; d) \subset \ov{\mc M} (V; d)
\eeqn
be the subset of elements whose evaluations lie in $S_i \subset X$. Then each $\ov{\mc M}_{S_i}(V; d)$ has expected dimension zero. Our purpose is then reduced to the definition of virtual counts 
\beqn
\# \ov{\mc M}_{S_i} (V; d) \in {\mb Q}
\eeqn
which correspond to the coefficients of a linear combination expression the pushforward of the expected VFC.

From now on, for notational purposes, we abbreviate $S_i$ by $S$ which is a smooth generator of  $H^{m_d-2}(V; {\mb Q})$. The more general case when $S_i$ are pseudocycles rather than compact cycles, can be derived without essential difficulty (see Remark \ref{rem520}). We need to define the virtual count of $\ov{\mc M}_S(V; d)$. We abbreviate 
\beqn
\ov{\mc M}:= \ov{\mc M}_S(V; d).
\eeqn

\subsection{Abstract charts}

We incorporate the abstract notions of Kuranishi charts of \cite{MW_3} and \cite{FOOO_2016}. Let ${\mc M}$ be a compact and Hausdorff space. 

\begin{defn}
An {\bf abstract} chart on ${\mc M}$ is a tuple 
\beqn
C = (U, E, S, \psi, F)
\eeqn
where
\begin{enumerate}

\item $U$ (the domain) is a locally compact Hausdorff normal topological space.

\item $E$ (the obstruction bundle) is a continuous map $\pi_E: E \to U$ from a locally compact Hausdorff space $E$ together with a zero section $0_E: U \to E$ which is a continuous map satisfying $\pi_E \circ 0_E = {\rm Id}_U$.

\item $S$ (the Kuranishi map) is a continuous map $S: U \to E$ satisfying $\pi_E \circ S = {\rm Id}_U$. Define $S^{-1}(0):= \{x \in U\ |\ S(x) = 0_E(x)\}$.

\item $F$ (the footprint) is an open subset of ${\mc M}$ ($F$ could be empty).

\item $\psi: S^{-1}(0) \to F$ (the footprint map) is a homeomorphism.
\end{enumerate}
\end{defn}

If $C = (U, E, S, \psi, F)$ is an abstract chart and $U' \subset U$ is an open subset, then one can restrict $C$ to $U'$ in a natural way and obtain another chart $C' = (U', E', S', \psi',F')$ where $E':= \pi_E^{-1}(U')$, $S': = S|_{U'}$, $\psi': = \psi|_{S^{-1}(0) \cap U'}$ and $F':= {\rm Im} \psi'$.

Now we can talk about coordinate changes. A {\bf topological embedding} from a topological space $W$ to another topological space $W'$ is a continuous map $\phi: W \to W'$ which is a homeomorphism onto its image (with respect to the subspace topology).

\begin{defn}\label{defn42}
Let $C_1, C_2$ be two charts. A {\bf coordinate change} from $C_1$ to $C_2$ is 
\beqn
T_{21} = (U_{21}, \hat\phi_{21})
\eeqn
where $U_{21} \subset U_1$ is an open subset, $\hat\phi_{21}: \pi_{E_1}^{-1}(U_{21}) \to E_2$ is a topological embedding satisfying 
\begin{enumerate}
\item $\hat\phi_{21}$ is a bundle map, namely, there is a topological embedding $\phi_{21}: U_{21} \to U_2$ (which must be unique) such that the following diagram commutes.
\beqn
\vcenter{ \xymatrix{   \pi_{E_1}^{-1}(U_{21}) \ar[dd]_-{\pi_{E_1}} \ar[rr]^-{\hat\phi_{21}} &  & E_2 \ar[dd]^{\pi_{E_2}} \\
& & \\
             U_{21} \ar@/^2pc/[uu]^{S_{E_1}} \ar@/_2pc/[uu]_{0_{E_1}} \ar[rr]_-{\phi_{21}} &  & U_2  \ar@/^2pc/[uu]_{ S_{E_2}} \ar@/_2pc/[uu]_{0_{E_2}} }}.
\eeqn

\item The above commutative diagram implies that $\phi_{21}$ maps $U_{21} \cap S_1^{-1}(0)$ into $S_2^{-1}(0)$. Then we require that the following diagram commutes.
\beqn
\vcenter{ \xymatrix{ U_{21} \cap S_1^{-1}(0) \ar[d]_{\psi_1} \ar[r]^-{\phi_{21}} & S_2^{-1}(0)  \ar[d]^{\psi_2} \\
                     {\mc M} \ar[r]^{{\rm Id}_{{\mc M}}} & {\mc M}  } }.
\eeqn

\item It follows from above that $\psi_1( U_{21} \cap S_1^{-1}(0)) \subset  F_1 \cap F_2$. We require that 
$$\psi_1(U_{21} \cap S_1^{-1}(0)) = F_1 \cap F_2.$$

\item The graph of $\phi_{21}$ viewed as a subset of $U_1 \times U_2$ is closed. Equivalently, it means if $x_n \in U_{21}$ converges to $x_\infty \in U_1$ and $y_n:= \phi_{21}(x_n)$ converges to $y_\infty \in U_2$, then $x_\infty \in U_{21}$ (and hence $y_\infty = \phi_{21}(x_\infty)$).\footnote{This condition is not required in other Kuranishi type approaches. However it can be achieved in the concrete construction and is very convenient for various operations.}
\end{enumerate}
\end{defn}

\begin{defn}(cf. \cite{FOOO_2016})\label{defn43}
An {\bf atlas} on $Z$ is a collection 
\beqn
{\mc A} = ({\mc I}, (C_I)_{I \in {\mc I}}, (T_{JI})_{I \preceq J})
\eeqn
where ${\mc I}$ is a finite partially ordered set, $C_I$ is a collection of charts on ${\mc M}$ and $T_{JI}$ is a collection of coordinate changes. We require the following conditions.
\begin{enumerate}
\item {\bf (Identity condition)} When $I = J$, $T_{II}$ is the identity coordinate change.

\item {\bf (Cocycle condition)} If $I\preceq J \preceq K$, denote $U_{KJI} = U_{KI} \cap \phi_{JI}^{-1}(U_{KJ}) \subset U_I$. Then we require that 
\beqn
\hat\phi_{KI}|_{U_{KJI}} = \hat \phi_{KJ} \circ \hat \phi_{JI}|_{U_{KJI}}.
\eeqn

\item {\bf (Overlapping condition)} If $\ov{F_I}\cap \ov{F_J} \neq \emptyset$, then either $I \preceq J$ or $J \preceq I$.

\item {\bf (Covering condition)} ${\mc M} = \bigcup_{I \in {\mc I}} F_I$.
\end{enumerate}
\end{defn}

The inductive construction of perturbations requires more conditions on the atlas. The condition was originally pointed out by D. Joyce. Define a binary relation $\curlyvee$ resp. $\tilde \curlyvee$ on the disjoint union  
\beqn
\bigsqcup_{I \in {\mc I}} U_I\ \ {\rm resp.}\ \ \bigsqcup_{I \in {\mc I}} E_I
\eeqn
as follows. We define $x \curlyvee y$ for $x \in U_I$ and $y\in U_J$ if one of the following holds
\begin{enumerate}
\item $I = J$ and $x = y$.

\item $I \preceq J$, $x \in U_{JI}$, and $y = \phi_{JI}(x)$.

\item $J \preceq I$, $y \in U_{IJ}$, and $x = \phi_{IJ}(y)$.
\end{enumerate}
We define $\tilde \curlyvee$ in a similar fashion. If $\curlyvee$ resp. $\tilde \curlyvee$ is an equivalence relation, then we can define the quotient space 
\beqn
|{\mc A}|:=  \left( \bigsqcup_{I \in {\mc I}} U_I \right) / \curlyvee\ \ {\rm resp.}\ \ |{\mc E}|:= \left( \bigsqcup_{I \in {\mc I}} E_I \right)/ \tilde \curlyvee.
\eeqn
called a {\bf virtual neighborhood} of ${\mc M}$ resp. {\bf virtual obstruction bundle}. They are equipped by default the quotient topology.

\begin{defn}
An atlas on ${\mc M}$ is called a 	{\bf good coordinate system} if
\begin{enumerate}
\item The relation $\curlyvee$ is an equivalence relation.

\item The virtual neighborhood $|{\mc A}|$ is a Hausdorff space.

\item The natural map $U_I \to |{\mc A}|$ for each $I$ is a topological embedding.
\end{enumerate}
\end{defn}

\begin{lemma}
Suppose ${\mc A}$ is a good coordinate system on ${\mc M}$. Then $\tilde \curlyvee$ is also an equivalence relation. Moreover, the bundle maps $\pi_I$ and $0_I$ define natural continuous maps 
\beqn
0_{\mc A}: |{\mc A}| \to |{\mc E}|,\ \pi_{\mc A}: |{\mc E}| \to |{\mc A}|.
\eeqn
\end{lemma}

\begin{proof}
The claim that $\tilde \curlyvee$ is an equivalence relation follows from the assumption that $\curlyvee$ is an equivalence relation and the cocycle condition. The claim that $0_{\mc A}$ and $\pi_{\mc A}$ are continuous follows from the definition of quotient topology.
\end{proof}



In applications good coordinate systems are often obtained by shrinking a given atlas. One is tempted to prove an abstract version of such a shrinking construction as \cite{FOOO_2016}, which would require more properties of charts such as metrizability. We will, however, prove such a shrinking lemma in the concrete situation in the next section.

\subsection{Stratified atlas}

We need the structure of certain ``stratifications'' on the moduli space for our construction. Here we give a precise definition. Let $(\Lambda, \leq)$ be a finite partially ordered sets equipped with a strictly increasing index function 
\beqn
{\rm ind}: \Lambda \to {\mb Z}.
\eeqn
A $\Lambda$-stratification on a topological space ${\mc M}$ is a decomposition
\beqn
{\mc M} = \bigsqcup_{\lambda \in \Lambda} {\mc M}_\lambda
\eeqn
by locally closed sets satisfying 
\beqn
\ov{\mc M}_\lambda \cap {\mc M}_{\lambda'} \neq \emptyset \Longrightarrow \lambda' \leq \lambda.
\eeqn
A continuous map between two $\Lambda$-stratified spaces are called stratified if it maps the $\lambda$-stratum of the domain into the $\lambda$-stratum of the target. 

\begin{defn}
Suppose ${\mc M}$ is compact Hausdorff and is $\Lambda$-stratified. A {\bf stratified chart} on ${\mc M}$ is a chart $C = (U, E, S, \psi, F)$ on ${\mc M}$ together with a $\Lambda$-stratification
\beqn
U = \bigsqcup_{\lambda \in \Lambda} U_\lambda
\eeqn
satisfying the following conditions.
\begin{enumerate}

\item Each stratum $U_\lambda \subset U$ is a topological orbifold and the restriction $E_\lambda:= E|_{U_\lambda}$ has the structure of an orbifold vector bundle. Moreover, 
\beqn
{\rm dim} U_\lambda - {\rm rank} E_\lambda = {\rm ind}(\lambda).
\eeqn

\item The map $\psi: S^{-1}(0) \to {\mc M}$ is stratified.

\end{enumerate}
\end{defn}

\subsection{Stratified transversality}

Now we prove the stratified transversality result in the relative setting. Recall that the notion of transversality in the $C^0$ category is the microbundle transversality. For maps into a vector space/sections of a vector bundle, there is a canonical microbundle to define transversality to the origin/zero section.

\begin{prop}\label{prop47}
Suppose $U$ is a locally compact, normal topological space with a $\Lambda$-stratification whose strata $U_\lambda$ are topological manifolds of dimension $d_\lambda$. Let $E \to U$ be a real vector bundle such that
\beq\label{eqn41}
\forall \lambda\ d_\lambda - {\rm rank} E = {\rm ind}(\lambda).
\eeq
Moreover, assume that 
\beqn
{\rm ind}(\lambda) \geq 0 \Longrightarrow \lambda\ {\rm is\ maximal}.
\eeqn
Let $Z \subset W \subset U$ be closed sets. Let $S: U \to E$ be a continuous section and let $t_Z: U \to E$ be another continuous section such that $S + t_Z$ is transverse over each stratum near $Z$. Then there exists a continuous section $t_W: U \to E$ such that
\begin{enumerate}
\item $t_Z = t_W$ over a neighborhood of $Z$.

\item $S + t_W$ is transverse over each stratum near $W$.

\end{enumerate}
Moreover, for any $\epsilon>0$ and any compact subset $Q \subset U$, we can require that 
\beqn
\| t_W \|_{C^0(Q)} \leq \| t_Z \|_{C^0(Q)} + \epsilon.
\eeqn
\end{prop}

\begin{proof}
We order the set $\Lambda$ increasingly as $\lambda_1, \ldots, \lambda_m$. Define 
\beqn
U^l:= \bigcup_{i \leq l} U_{\lambda_i}
\eeqn
which is a closed set. We prove inductively that one can find a section $t_W^l: U \to E$ such that $t_Z = t_W$ near $Z$ and $S + t_W$ is transverse over each stratum near $W \cap U^l$. For the base case, as $U_{\lambda_1}$ is a manifold, by the topological transversality extension theorem of Kirby--Siebenmann \cite{Kirby_Siebenmann} and Quinn \cite{Quinn_1982}\cite{Quinn}\cite{Freedman_Quinn}, one can find a section $t_W: U_{\lambda_1} \to E$ which coincides with $t_Z$ near $Z$ and $S + t_W$ is transverse near $W$. Then there is a neighborhood $O_Z$ of $Z$ such that $t_Z$ and $t_W$ together define a continuous section
\beqn
t_W^1: O_Z \cup U_{\lambda_1} \to E.
\eeqn
As $U$ is normal, there is a smaller neighborhood $O_Z'$ of $Z$ whose closure is contained in $O_Z$. Then by the Tietze Extension Theorem, one can extend $t_W^1|_{\ov{O_Z'} \cup U_{\lambda_1}}$ to a continuous map on $U$, still denoted by $t_W^1$. We only need to verify that $t_W^1$ is transverse on every stratum near $W \cap U_{\lambda_1}$. Indeed, as its restriction is transverse over $U_{\lambda_1}$, by the index condition \eqref{eqn41}, unless $\lambda_1$ is a top stratum (in which case our proof is already finished), it follows that $t_W^1$ is nowhere vanishing near $U_{\lambda_1}$. Hence in particular, $t_W^1$ is transverse near $W \cap U_{\lambda_1}$. 

Suppose we have constructed $t_W^l: U \to E$ satisfying the induction hypothesis. Then we can replace $Z$ by $Z \cup U^l$ and replace $W$ by $W \cup U^l$. Then $t_W^l$ is transverse over the stratum $U_{\lambda_{l+1}}$ near $Z$. One can then construct as above a new map $t_W^{l+1}$ defined near $Z$ and over $U^{l+1}$ which coincides with $t_W^l$ near $Z$ and transverse over $U_{\lambda_{l+1}}$. The Tietze Extension Theorem allows one to extend it to $U$, which is still nowhere vanishing near $U^{l+1}$. Hence the induction can be carried out. 
\end{proof}

\subsection{Multimaps and multisections}

Let $A$ be an arbitrary nonempty set and $l \geq 1$ be an integer. Denote 
\beqn
S^l(A):=  \underbrace{A \times \cdots \times A}_{l\ {\rm copies}}/ S_l
\eeqn
where the symmetry group $S_l$ acts by permutations. An element of $S^l(A)$ is denoted by 
\beqn
[a_1,\ldots, a_l].
\eeqn
For any other positive integer $k$, there is a natural map 
\beqn
m_k: S^l(A) \to S^{kl}(A),\ [a_1, \ldots, a_l] \mapsto [ \underbrace{a_1, \ldots, a_1}_{k\ {\rm copies}}, \ldots, \underbrace{a_l, \ldots, a_l}_{k\ {\rm copies}}].
\eeqn
If $U$ is another set, then an $l$-multimap from $U$ to $A$ is a map 
\beqn
f: U \to S^l(A).
\eeqn
We say that an $l_1$-multimap $f_1: U \to S^{l_1}(A)$ and an $l_2$-multimap $f_2: U \to S^{l_2}(A)$ are {\bf equivalent} if $m_{l_2} \circ f_1 = m_{l_1} \circ f_2$. This relation is obviously transitive and hence an equivalence relation. An equivalence class is simply called a {\bf multimap} from $U$ to $A$, denoted by 
\beqn
f: U \overset{m}{\to} A.
\eeqn 	
Moreover, if $A$ and $U$ are both topological spaces, then the symmetric products $S^l(A)$ inherits natural topologies from $A$ and we have the notion of continuous multimaps. 

A continuous multimap $f: U \overset{m}{\to} A$ is called {\bf liftable} if there exist $l \geq 1$ and $l$ continuous maps $f_1, \ldots, f_l: U \to A$ (called {\bf branches} of $f$) such that 
\beqn
f(x) = [f_1(x), \ldots, f_l(x)],\ \forall x\in U.
\eeqn

From now on all multimaps are assumed to be between topological spaces and are assumed to be continuous without further clarification. 

Now we consider the typical situation of multisections in a chart. Let $C = (U, E, S, \psi, F)$ be a certain abstract chart on some locally compact and Hausdorff topological space ${\mc M}$. A multisection of $E$ is a multimap $t: U \overset{m}{\to} E$ such that $\pi_E \circ t = {\rm Id}_U$. In applications, the chart is usually constructed in a specific way.

\begin{defn}
Let $\Gammait$ be a finite group. An abstract chart $C = (U, E, S, \psi, F)$ is called a {\bf $\Gammait$-quotient} of $({\bm U}, {\bm E}, {\bm S})$ if ${\bm U}$ is a locally compact and Hausdorff topological space with a $\Gammait$-action, ${\bm E}$ is a real representation ${\bm E}$ of $\Gammait$, and ${\bm S}: {\bm U} \to {\bm E}$ is a $\Gammait$-equivariant map such that (denoting $\Gammait$-orbits of an element $e$ by $\Gammait\cdot e$)
\beqn
U = {\bm U}/\Gammait,\ E = ({\bm U}\times {\bm E})/\Gammait,\ \pi_E( \Gammait \cdot ({\bm x}, {\bm e})) = \Gammait \cdot {\bm x},\ 0_E(\Gammait \cdot {\bm x}) = \Gammait\cdot ({\bm x}, 0)
\eeqn
and 
\beqn
S(\Gammait \cdot {\bm x}) = \Gammait \cdot ({\bm x}, {\bm S}({\bm x})).
\eeqn
\end{defn}

\begin{defn}\label{defn49}
Suppose the chart $C$ is the $\Gammait$-quotient of $({\bm U}, {\bm E}, {\bm S})$ and let ${\bm Q} \subset {\bm U}$ be an arbitrary $\Gammait$-invariant subset whose quotient is $Q\subset U$. A multisection $t: Q \overset{m}{\to} E$ is called  $\Gammait$-liftable if there exist $l \geq 0$ and continuous maps ${\bm t}_i: {\bm Q} \to {\bm E}$ such that the associated $l$-multimap $[{\bm t}_1, \ldots, {\bm t}_l]: {\bm Q} \to S^l({\bm E})$ is $\Gammait$-equivariant and 
\beqn
t( \Gammait\cdot {\bm x}) = \Gammait\cdot [({\bm x}, {\bm t}_1({\bm x})), \ldots, ({\bm x}, {\bm t}_l({\bm x}))].
\eeqn
The $l$-tuple $({\bm t}_1, \ldots, {\bm t}_l)$ is called a $\Gammait$-lift of $t$. 

Two $\Gammait$-liftable multisections $t, t': Q \overset{m}{\to} E$ are said to agree {\bf in the strong sense} if there exist $l\geq 1$, $\Gammait$-lifts $({\bm t}_1, \ldots, {\bm t}_l)$, $({\bm t}_1', \ldots, {\bm t}_l')$, and a permutation $\sigma \in S_l$ such that 
\beqn
{\bm t}_i = {\bm t}_{\sigma(i)}'.
\eeqn

\end{defn}

We have the following transversality extension lemma for such multisections. 



\begin{lemma}\label{lemma410}
Suppose the chart $C$ is the $\Gammait$-quotient of $({\bm U}, {\bm E}, {\bm S})$ and suppose that ${\bm U}$ satisfies assumptions of Proposition \ref{prop47}. Let $Z \subset W\subset U$ be closed subsets. Suppose $t$ is a continuous $\Gammait$-liftable multisection of $E$ defined near $W$ such that $S + t$ is transverse over each stratum near $Z$. Then there exists a $\Gammait$-liftable continuous multisection $t': U \overset{m}{\to} E$ such that $t$ and $t'$ agree (in the strong sense) near $Z$ and such that $S + t'$ is transverse over each stratum over $U$. Moreover, we may require that for all $\epsilon>0$ and all compact subsets $Q \subset Z$ there holds
\beqn
\| t' \|_{C^0(Q)} \leq \| t \|_{C^0(W \cap Q)} + \epsilon.
\eeqn
\end{lemma}

\begin{proof}
Let ${\bm Z}, {\bm W} \subset {\bm U}$ be the unions of $\Gammait$-orbits in $Z$ and $W$ respectively. For each branch ${\bm t}_i$ of $t$, by Proposition \ref{prop47}, one can find a continuous map ${\bm t}_i': {\bm U} \to {\bm E}$ such that ${\bm t}_i = {\bm t}_i'$ near ${\bm Z}$ and such that ${\bm S} + {\bm t}_i'$ is transverse over each stratum. The collection of ${\bm t}_i'$ may not be $\Gammait$-equivariant, but one can symmetrize them to get an $\Gammait$-equivariant liftable multimap possibly increasing the number of branches. The symmetrization preserves transversality.
\end{proof}

\begin{rem}
Although it is possible, we refrain from stating and proving a general theorem about constructing a virtual fundamental cycle or even a virtual count in the abstract setting as it needs more assumptions. Instead, we give the VFC construction next section in the concrete situation where many conditions needed for constructing the perturbations are more natural to have. 
\end{rem}

\section{Good Coordinate Systems on Moduli Spaces and Perturbations}\label{section5}

We fix a degree $d\in H_2^K(V; {\mb Z})$. Recall that $\ov{\mc M}(V, d)$ is the moduli space of stable degree $d$ pointlike instantons in $V$ and $\ov{\mc M}_S(V, d) \subset \ov{\mc M}(V, d)$ is the subset of configurations whose evaluations at infinity lie in the (smooth) cycle $S \subset X$. $\ov{\mc M}_S(V, d)$ has virtual dimension zero. The purpose of this section is to construct a good coordinate system on $\ov{\mc M}_S(V, d)$ together with a suitable kind of perturbation and to define a virtual count.

\subsection{Moduli space of marked scaled curves and its unfoldings}

The structure of the moduli space of marked scaled curves has been studied in \cite{Mau_Woodward_2010}. Let $k \geq 1$ and let $\ov{\mc M}{}_k^1$ be the moduli space of stable marked scaled curves with $k$ (ordered) marked points. This moduli space can be viewed as a natural compactification of the configuration space of $k$ ordered distinct points of ${\mb C}$ modulo translation. By \cite[Theorem 1.2]{Mau_Woodward_2010}, $\ov{\mc M}{}_k^1$ admits the structure of a complex projective variety with toric singularities. One can see that there is a universal family 
\beq\label{eqn51}
\ov{\mc M}{}_{k+1}^1 \to \ov{\mc M}{}_k^1
\eeq
where the projection is defined by forgetting the last marked point and stabilization. 

Now we would like to give a more concrete local description of the universal family \eqref{eqn51}.  Fix a point $p \in \ov{\mc M}{}_k^1$ of a certain type $\Gamma$. For the purpose of this paper, we assume that $V_\Gamma^0 = \emptyset$ as before. We first describe deformations which fix the combinatorial form. Choose a representative ${\mc C}$ of the isomorphism class $p$. For each vertex $v \in V_\Gamma$, the component ${\mc C}_v \subset {\mc C}$ can be identified with the affine line ${\mb C}$ with special points
\beqn
(w_e)_{e \in E_v},\  w_e \in {\mb C}.
\eeqn
Here 
\beqn
E_v:= \left\{ \begin{array}{cc} \{ e \in E_\Gamma^+\ |\ v(e) = v \},\ &\ {\rm if}\ v \in V_\Gamma^1,\\
                                \{ e \in E_\Gamma\ |\ v(e) = v\},\ &\ {\rm if}\ v \in V_\Gamma^\infty\end{array}\right.,
\eeqn
i.e., the set of oriented finite or semi-infinite edges ending at $v$. Two such representations $(w_e)_{e\in E_v}$ and $(w_e')_{e\in E_v}$ are isomorphic if they differ by a translation resp. affine linear transformation if $v \in V_\Gamma^1$ resp. $v \in V_\Gamma^\infty$. Then there is a complex vector space $V_{{\rm def}, v}$ parametrizing infinitesimal deformations of the marked curve ${\mc C}_v$ with 
\beqn
{\rm dim} V_{{\rm def}, v} = \left\{ \begin{array}{cc} \# E_v -1,\ &\ {\rm if}\ v \in V_\Gamma^1,\\
                                                      \# E_v -2,\ &\ {\rm if}\ v \in V_\Gamma^\infty.\end{array}\right.
\eeqn
Let $V_{{\rm def}, v}^\epsilon \subset V_{{\rm def},v}$ be a sufficiently small neighborhood of the origin. Then one can choose a smooth function  
\beq\label{eqn52}
V_{{\rm def}, v}^\epsilon \to {\mb C}^{E_v},\ \eta \mapsto (w_{e, \eta})_{e\in E_v}
\eeq
such that $w_{e, 0} = w_e$ and the induced map from $V_{{\rm def}, v}^\epsilon$ to the domain moduli ${\mc M}^{{\mf s}(v)}_{\# E_v}$ is a homeomorphism onto an open neighborhood of $[{\mc C}_v]$.\footnote{When ${\mf s}(v) = \infty$, ${\mc M}_k^\infty$ is the moduli space of configurations of $k$ distinct points in ${\mb C}$ modulo affine linear transformation, or equivalently, the main stratum of the moduli space of genus zero stable curves with $k+1$ marked points.}

Now we describe deformations which resolve the nodes. For each node (edge) $e \in E_\Gamma$ one can define a space of gluing parameters $\Delta_e \cong {\mb C}$. The gluing parameters are not independent, but should satisfy a ``balanced condition.''
\begin{defn}
Let $\Gamma$ be a stable combinatorial type of marked scaled curves satisfying $V_\Gamma^0 = \emptyset$. A tuple of gluing parameters $(\zeta_e) \in \prod_{e\in E_\Gamma} \Delta_e$ is said to satisfy the {\bf balanced condition} if for any two paths 
\beqn
v_1 \overset{e_1}{\succ} v_2 \overset{e_2}{\succ} \cdots \overset{e_{k-1}}{\succ} v_k,\ v_1' \overset{e_1'}{\succ} v_2' \overset{e_2'}{\succ} \cdots \overset{e_{l-1}'}{\succ} v_l',\ v_k = v_l' \in V_\Gamma^\infty,\ v_1, v_1' \in V_\Gamma^1
\eeqn
there holds
\beqn
\zeta_{e_1} \cdots \zeta_{e_{k-1}} = \zeta_{e_1'} \cdots \zeta_{e_{l-1}'} \in {\mb C}.
\eeqn
\end{defn}

Define
\beqn
V_{\rm res}:= V_{{\rm res}, p}:= \Big\{ ( \zeta_e)_{e\in E_\Gamma}\ |\ (\zeta_e)\ {\rm satisfies\ the\ balanced\ condition}\Big\}
\eeqn
which is an affine variety. 

Now one can explicitly construct a local universal family. The general case can be derived from the case when all gluing parameters are nonzero. We fix the point $p\in \ov{\mc M}_k^1$, the representative ${\mc C}$, and the family of functions $w_{e, \eta}$ (see \eqref{eqn52}) as above.  For each semi-infinite edge $e_i$, suppose the path connecting $e_i$ to the root vertex is 
\beqn
v_1 \overset{e_1}{\succ} v_2 \overset{e_2}{\succ} \cdots \overset{e_l}{\succ} v_\infty.
\eeqn
Then for any small deformation parameter $\eta$ and small resolution parameter $\zeta$, define
\beq\label{eqn53}
w_{i, \eta, \zeta} = w_{i, \eta} + \frac{1}{\zeta_{e_1}} w_{e_1, \eta} + \frac{1}{\zeta_{e_1} \zeta_{e_2}} w_{e_2, \eta} + \cdots + \frac{1}{\zeta_{e_1} \cdots \zeta_{e_l}} w_{e_l, \eta}.
\eeq
Let ${\mc C}_{\eta, \zeta}$ be the curve ${\mc C}$ with the above list of marked points. This defines a map 
\beq\label{eqn54}
\begin{array}{ccccc}
\displaystyle V_{\rm def}^\epsilon \times V_{\rm res}^\epsilon & \to & \displaystyle \bigsqcup_{\Gamma \preceq \Pi} \prod_{v\in V_\Pi} {\mb C}^{E_v} & \to & \ov{\mc M}_k^1,\\
(\eta, \zeta) & \mapsto & {\bm w}_{\eta, \zeta} = (w_{i, \eta, \zeta})_{1\leq i \leq k} & \mapsto & [{\mc C}_{\eta, \zeta}].
\end{array}
\eeq
Then there holds the following fact.
\begin{lemma}(cf. \cite{Mau_Woodward_2010})
For each stable scaled marked curve ${\mc C}$, for all small $\epsilon>0$, the map \eqref{eqn54} is a homeomorphism between $V_{\rm def}^\epsilon \times V_{\rm res}^\epsilon$ and an open neighborhood of $[{\bm w}]$ in $\ov{\mc M}_k$.
\end{lemma}

One can make this family of marked curves ${\mc C}_{\eta, \zeta}$ into a local universal family. Define 
\beqn
{\mc U}:= \bigsqcup_{(\eta, \zeta)} {\mc C}_{\eta, \zeta}.
\eeqn
In fact one can define a canonical topology on ${\mc U}$ and identify ${\mc U}$ as an open subset of $\ov{\mc M}_{k+1}^1$ so that one has the commutative diagram
\beqn
\vcenter{ \xymatrix{ {\mc U} \ar[r]  \ar[d] & \ov{\mc M}_{k+1}^1 \ar[d]\\
           {\mc V} \ar[r]   & \ov{\mc M}_k^1 } }.
\eeqn
and so that the natural projection $\pi: {\mc U} \to {\mc V}$ is continuous and the sections of marked points $w_i: {\mc V} \to {\mc U}$ are continuous. There is also a closed subset ${\mc U}^* \subset {\mc U}$ corresponding to markings or nodes. This gives us a tuple 
\beqn
({\mc U}, {\mc V}, 0, \pi, w_1,\ldots, w_k)
\eeqn
which we call a {\bf local universal unfolding} of the marked stable scaled curve ${\mc C}$. Notice that this family only depends on the choices of the neighborhoods $V_{\rm def}^\epsilon \subset V_{\rm def}$ and $V_{\rm res}^\epsilon \subset V_{\rm res}$ and the family of markings \eqref{eqn52}. 

\subsubsection*{Unordered marked points}

For the paper of the construction, we also need to consider the case when the markings are unordered. In this situation the marked curve may have nontrivial automorphisms which permutes the markings. Indeed, let ${\mc C}$ be a marked stable scaled curve where the markings are unordered. Let $\Gammait$ be the (finite) automorphism group. Choose an arbitrary order of the markings, which gives a marked stable scaled curve $\hat {\mc C}$. Then $\Gammait$ permutes the ordered marked points. One still have the space of deformation parameters $V_{\rm def}$ and the space of resolution parameters $V_{\rm res}$. The group $\Gammait$ acts linearly on $V_{\rm def}$ and $V_{\rm res}$. We can take the neighborhoods of origins $V_{\rm def}^\epsilon$ and $V_{\rm res}^\epsilon$ to be $\Gammait$-invariant. Then the total space of the local universal family ${\mc U} \to {\mc V}:= V_{\rm def}^\epsilon \times V_{\rm res}^\epsilon$ also has the natural $\Gammait$-action whose restrictions to a fibre ${\mc C}_{\eta, \zeta}$ is an isomorphism
\beqn
\gamma: {\mc C}_{\eta, \zeta} \to {\mc C}_{\gamma \eta, \gamma \zeta}
\eeqn 
which permutes the markings $(w_{i, \eta, \zeta})_{1 \leq i \leq k}$. In other words, the sections $w_1, \ldots, w_k: {\mc V}\to {\mc U}$ given by the formula \eqref{eqn53} defines a $\Gammait$-invariant multivalued map 
\beqn
(\eta, \zeta) \mapsto [ w_{1, \eta, \zeta}, \ldots, w_{k, \eta, \zeta}].
\eeqn

We will talk about functions and sections over the total space ${\mc U}$. It is easy to see from the construction that the complement of the set ${\mc U}^*$ of markings and nodes is a smooth manifold and the smooth part of each fibre has a well-defined complex structure. Then we define
\beqn
C^\infty ({\mc U}) = \Big\{ f \in C^0( {\mc U})\ |\ f|_{{\mc U}^*} {\rm is\ smooth} \Big\}.
\eeqn
There is also the bundle  
\beqn
\Lambda^{0,1}_{{\mc U}/ {\mc V}} \to {\mc U}^*
\eeqn
whose fibre at a point $p \in {\mc U}^*$ is the $\Lambda^{0,1}$-space in the fibre direction. For smooth bundles $E \to {\mc U}^*$ such as $\Lambda^{0,1}_{{\mc U}/ {\mc V}}$ which is only defined over ${\mc U}^*$, we can define the space 
\beqn
C^\infty_c({\mc U}, E)
\eeqn
which is the space of smooth sections whose fibrewise supports are compact and disjoint from ${\mc U}^*$.

Lastly we name a concept which will be useful in the construction of charts.

\begin{defn}\label{defn53}
Let ${\mc C}$ be a marked stable scaled curve with unordered markings. A {\bf resolution datum} consists of 
\begin{enumerate}

\item An ordering of the marked points making an marked stable scaled curve $\hat{\mc C}$. Then one has the space of deformation parameters $V_{\rm def}$ and the space of gluing parameters $V_{\rm res}$. The automorphism group $\Gammait$ acts on both of them. 

\item An unfolding ${\mc U} \to {\mc V}:=V_{\rm def}^\epsilon \times V_{\rm res}^\epsilon$.

\item A $\Gammait$-invariant precompact open subset $U \subset {\mc C}$ disjoint from special points and a smooth $\Gammait$-equivariant trivialization 
\beq\label{eqn55}
U \times {\mc V} \hookrightarrow {\mc U} \setminus {\mc U}^*
\eeq
of the complement of a small neighborhood of ${\mc U}^*$. 
\end{enumerate}
\end{defn}

From our construction one can see that resolution data exist.

\subsection{Stabilizing the domain}\label{subsection52}

Let ${\mc C}$ be a scaled curve of domain type $\Gamma$ without markings. Assume $V_\Gamma^0 = \emptyset$. This is a typical domain for a stable pointlike instanton. In order to construct local charts, the usual procedure is to stabilize the domain. We can do it as follows. For each $v \in V_\Gamma^1$, we add a marking to ${\mc C}_v$ so that this domain is stable. For each $v \in V_\Gamma^\infty$ which is unstable (meaning that it has less than three special points). Then we need to add one or more special points to stabilize this component. Suppose we need to add two. Then to the tree $\Gamma$, we add two new vertices $v', v''  \in V_{\Gamma'}^1$ to obtain a new scaled tree $\Gamma'$, and add two tails $e', e''$ attached to $v'$ and $v''$ respectively. The extra vertices $v'$ and $v''$ are necessary because we require that tails (markings) must be attached to vertices corresponding to pointlike instantons but not holomorphic spheres. To the domain ${\mc C}$, we add two new extra components ${\mc C}_{v'}$ and ${\mc C}_{v''}$ together with two new markings $w_{e'} \in {\mc C}_{v'}$ and $w_{e''}\in {\mc C}_{v''}$. This stabilizing procedure is illustrated in Figure \ref{figure1}.

\begin{center}
\begin{figure}[h]
\label{figure1}
\includegraphics[scale=1]{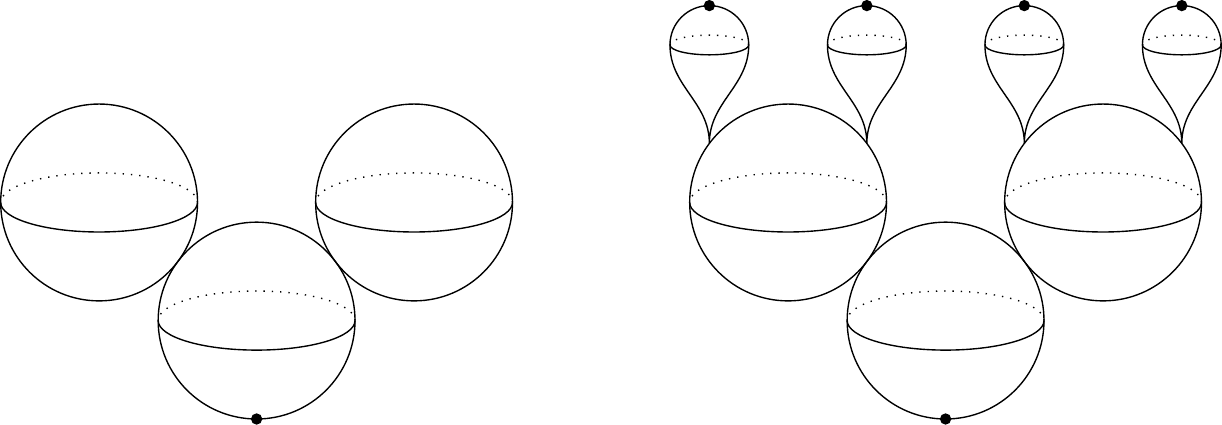}
\caption{A stable holomorphic sphere with three components is stabilized by adding four markings and four ghost components.}
\end{figure}
\end{center}

We see that this stabilizing operation can be made respect the automorphism group. Indeed, let $\Gammait$ be the finite automorphism group of the stable pointlike instanton $({\mc C}, {\bm y})$. Then one can stabilize ${\mc C}$ to a new domain curve ${\mc C}'$ such that $\Gammait$ acts on ${\mc C}'$ making ${\bm y}$ invariant.

\subsection{Thickened moduli spaces}

To regularize the moduli space $\ov{\mc M}_S(V, d)$ one needs to thickened it locally by adding obstructions. 

\begin{defn}\label{defn54}
For $p \in {\mc M}_S(V, d)_\lambda \subset \ov{\mc M}_S(V, d)$, a {\bf thickening datum} $\alpha$ centered at $p$ is a tuple 
\beqn
\alpha = \left( {\mc C}_\alpha, {\bm u}_\alpha, {\bm y}_\alpha, {\bm r}_\alpha, {\bm E}_\alpha, \iota_\alpha, {\bm H}_\alpha \right)
\eeqn 
where 
\begin{enumerate}
\item ${\mc C}_\alpha$ is a smooth or nodal scaled curve whose domain type is determined by the map type $\lambda$.

\item ${\bm u}_\alpha$ is a smooth scaled pointlike instanton over ${\mc C}_\alpha$ such that $({\mc C}_\alpha, {\bm u}_\alpha)$ represents $p$ (after collapsing unstable components).

\item ${\bm y}_\alpha$ is an unordered list of points on ${\mc C}_\alpha$ that makes $({\mc C}_\alpha, {\bm y}_\alpha)$ stable such that if $\Gammait_\alpha$ is the automorphism group of $({\mc C}_\alpha, {\bm u}_\alpha)$, then the image of the natural map $\Gammait_\alpha \to {\rm Aut}({\mc C}_\alpha)$ is contained in ${\rm Aut}({\mc C}_\alpha, {\bm y}_\alpha)$. 

\item ${\bm r}_\alpha$ is a resolution datum for $({\mc C}_\alpha, {\bm y}_\alpha)$, which includes an explicit universal unfolding 
\beqn
\pi_\alpha: {\mc U}_\alpha \to {\mc V}_\alpha.
\eeqn
Notice that the automorphism group $\Gammait_\alpha$ acts on the unfolding via the map $\Gammait_\alpha \to {\rm Aut}({\mc C}_\alpha, {\bm y}_\alpha)$. 

\item ${\bm E}_\alpha$ is a finite-dimensional real representation of $\Gammait_\alpha$.

\item $\iota_\alpha$ is a $\Gammait_\alpha$-equivariant linear map 
\beqn
\iota_\alpha: {\bm E}_\alpha \to C_c^\infty({\mc U}_\alpha \times V, \Lambda^{0,1}_{{\mc U}_\alpha/{\mc V}_\alpha}\otimes TV)^K
\eeqn
such that the restriction of $\iota_\alpha({\bm E}_\alpha)$ to the central fibre of ${\mc U}_\alpha \to {\mc V}_\alpha$ is transverse to the image of $D_\alpha$ and its restriction to each fibre are generated by sections with supports contained in $U_\alpha$, where $U_\alpha$ is part of the information included in the resolution datum ${\bm r}_\alpha$ (see Definition \ref{defn53} and \eqref{eqn55}).

\item ${\bm H}_\alpha \subset V$ is a collection of smooth embedded codimension two submanifold (not necessarily $K$-invariant) whose members $H_{\alpha, y}$ are indexed by $y \in {\bm y}_\alpha$, satisfying the following conditions: 1) if $y\in {\bm y}_\alpha$ is on a nonconstant instanton component $\Sigma_v$, then the map $u_v: \Sigma_v \to V$ intersects $H_{\alpha, y}$ transversely at $y$; 2) $y\in {\bm y}_\alpha$ is on a ghost instanton component attached to a nonconstant holomorphic sphere component $u_v: \Sigma_v \cong S^2 \to X$ at $y'\in \Sigma_v$, then $H_{\alpha, y}$ is $K$-invariant and intersects the level set $\mu^{-1}(0)$ transversely (hence descends to an embedded codimension two submanifold $\ov{H}_{\alpha, y}\subset X$, and the holomorphic sphere $u_v$ intersects $\ov{H}_{\alpha, y}$ transversely at $y$.
\end{enumerate}
\end{defn}

Now we define the notion of thickened solutions and thickened moduli spaces. To define the normalization condition given by the not-necessarily $K$-invariant hypersurfaces $H_\alpha$, one needs to specify gauge slices. Remember that the $K$-action in the semi-stable locus $V^{\rm ss} \subset V$ is free. Then for each $x \in V^{\rm ss}$, the subspace  
\beqn
T_x Kx:= \{ {\mc X}_a(x)\ |\ a \in {\mf k}\} \subset T_x V
\eeqn
has dimension equal to ${\rm dim} K$. Then for $\delta>0$ sufficiently small, 
\beqn
W_x^\delta:= \exp_x \left( \{ v \in (T_x Kx)^\bot\ |\ \| v \| < \delta \} \right) \subset V^{\rm ss}
\eeqn
is an embedded submanifold which is a local slice of the $K$-action through the point $x$. We call  $W_x^\delta$ the {\bf standard $K$-slice} through $x$. 

\begin{defn}\label{defn55}
Let $I$ be a finite set of thickening datum. An {\bf $I$-thickened solution} is a quadruple
\beqn
\Big( {\mc C}, {\bm u}, ({\bm y}_\alpha)_{\alpha \in I}, (\phi_\alpha)_{\alpha \in I}, (e_\alpha)_{\alpha \in I} \Big)
\eeqn
where 
\begin{enumerate}

\item ${\mc C}$ is a smooth or nodal scaled curve and for each $\alpha \in I$, ${\bm y}_\alpha$ is an unordered list of marked points. We require that for each $\alpha \in I$, $({\mc C}, {\bm y}_\alpha)$ is stable.

\item ${\bm u} = (u_v)_{v\in V_\Gamma}$ is a smooth gauged map (not a gauge equivalence class) on ${\mc C}$ such that for each $y \in \cup_{\alpha \in I} {\bm y}_\alpha$, ${\bm u} (y)$ is contained in the semi-stable locus of $V$. 

\item For each $\alpha$, $\phi_\alpha: {\mc C} \to {\mc C}_{\phi_\alpha} \cong {\mc C}_{\alpha, \eta, \zeta}$ is an isomorphism for some fibre ${\mc C}_{\phi_\alpha} \subset {\mc U}_\alpha$ such that $\phi_\alpha({\bm y}_\alpha) = {\bm y}_{\alpha, \eta, \zeta} = {\bm y}_{\phi_\alpha}$. For each $y \in {\bm y}_\alpha$, denote $y_{\phi_\alpha}\in {\mc C}_{\phi_\alpha}$ the corresponding point and $H_{y_{\phi_\alpha}} \subset {\bm H}_\alpha$ the corresponding component through which $u_\alpha$ passes at $y_{\phi_\alpha}$.  

\item $e_\alpha \in {\bm E}_\alpha$ is a vector and denote by 
\beqn
Z_\alpha:= {\rm supp} \left( \iota_\alpha(e_\alpha) \circ \phi_\alpha \right).
\eeqn
Then ${\bm u}(Z_\alpha)$ is contained in the semi-stable locus $V^{\rm ss}\subset V$. 

\end{enumerate}
These data also need to satisfy the following additional conditions. 

\begin{enumerate}
\item For each $y \in {\bm y}_\alpha$, over a small disk $B_r(y)$ there exists (must be unique) gauge transformation $g$ such that for each $z \in B_r(y) \subset {\mc C}$ such that 
\beqn
g(z)^{-1} u(z) \in W_{u_\alpha(\phi_\alpha(z))}^\delta
\eeqn
for some $\delta>0$. Moreover, $g^{-1} u$ intersect $H_{y_{\phi_\alpha}}$ transversely at $y$. 

\item We require that the quadruple must satisfy the $I$-perturbed gauged Witten equation, i.e.,
\beqn
{\mc F}_{\mc C}({\bm u}) + \sum_{\alpha \in I} \iota_\alpha(e_\alpha)\circ \phi_\alpha = 0.
\eeqn
\end{enumerate}

Two $I$-thickened solutions $({\mc C}, {\bm u}, ({\bm y}_\alpha), (\phi_\alpha), (e_\alpha))$ and $({\mc C}', {\bm u}', ({\bm y}_\alpha'), (\phi_\alpha'), (e_\alpha'))$ are {\it isomorphic} if 
\begin{enumerate}

\item for each $\alpha \in I$, $e_\alpha = e_\alpha'$;

\item there exists a domain isomorphism $\rho: {\mc C} \to {\mc C}'$ such that for all $\alpha \in I$, the following diagram commutes
\beqn
 \xymatrix{  {\mc C} \ar[r]^{\phi_\alpha} \ar@/^2pc/[rr]^{\rho} & {\mc U}_\alpha        &  {\mc C}' \ar[l]_{\phi_\alpha'} }
\eeqn

\item there is a gauge transformation over the trivial bundle over ${\mc C}$ such that 
\beqn
{\bm u} = g^* \rho^* {\bm u}'.
\eeqn
\end{enumerate}
Here ends the definition.
\end{defn}

For any finite set $I$ of thickening data, define ${\mc M}_I$ to be the set of isomorphism classes of $I$-thickened solutions. One can define a Gromov-type topology on ${\mc M}_I$ which is locally compact and Hausdorff. Moreover, ${\mc M}_I$ is naturally stratified by combinatorial types as 
\beqn
{\mc M}_I = \bigsqcup_{\lambda \in \Lambda} {\mc M}_{I, \lambda}.
\eeqn
We would like to have nice topological properties of ${\mc M}_I$, such as metrizability, to make further argument. However, it is difficult to prove it directly (cf. \cite[Second Erratum]{McDuff_Salamon_2004}). Instead, we can easily verify a weaker condition.

\begin{lemma}\label{lemma56}
Each stratum ${\mc M}_{I, \lambda} \subset {\mc M}_I$ is metrizable.
\end{lemma}

\begin{proof}
Notice that each facet $V_{{\rm res}, \lambda}$ of the space of gluing parameters is diffeomorphic to an open subset of a Euclidean space. Hence each ${\mc M}_{I, \lambda}$ is a subset of a certain Banach space, hence is metrizable.
\end{proof}

\begin{rem}\label{rem57}
Using the same argument one can see that each stratum of the original moduli space $\ov{\mc M}_S(V, d)$ is also metrizable. 
\end{rem}

\begin{cor}\label{cor58}
For each compact subset $Z \subset {\mc M}_I$, there exists a sequence of shrinking neighborhoods of $Z$, i.e., a sequence of open sets $W_n \subset {\mc M}_I$ containing $Z$ such that
\beqn
Z = \bigcap_{n \geq 0} W_n.
\eeqn
\end{cor}

\begin{proof}
Order elements of $\Lambda$ increasingly as $\lambda_1, \ldots, \lambda_m$. For each $k$, denote 
\begin{align*}
&\ Z_k:= Z \cap {\mc M}_{I, \lambda_k},\ &\ Z_{\leq k}:= \bigcup_{l \leq k} Z_l.
\end{align*}
Then each $Z_{\leq k}$ is compact. We prove inductively that each $Z_{\leq k}$ has a shrinking sequence of neighborhoods inside ${\mc M}_{I, \leq k}:= \bigcup_{\lambda_l \leq \lambda_k} {\mc M}_{I, \lambda_l}$. For the initial stage $\lambda_1$, which is a minimal stratum, because ${\mc M}_{I, \lambda_1}$ is metrizable and $Z_1$ is compact, there is a shrinking sequence of neighborhoods of $Z_1$ inside ${\mc M}_{I, \leq 1}$. Suppose the claim is true for $k - 1$ for some $k \leq m$. We would like to construct a sequence of shrinking neighborhoods of $Z_{\leq k}$ inside ${\mc M}_{I, \leq k}$.

First, by induction hypothesis, one can find a sequence of shrinking neighborhoods $W_{k-1,n} \subset {\mc M}_{I, \leq k-1}$ of $Z_{\leq k-1}$. In a thickened moduli space, the stratification is determined by the vanishing of certain gluing parameters. Then one can choose a sequence of open neighborhoods $U_{k, n} \subset {\mc M}_{I, \leq k}$ of ${\mc M}_{I, \leq k-1}$ such that 
\beqn
U_{k, n+1} \subset U_{k, n},\ {\mc M}_{I, \leq k-1} = \bigcap_{n \geq 1} U_{k, n+1}.
\eeqn
Then we can find a sequence of neighborhoods $W_{k, n}' \subset {\mc M}_{I, \leq k}$ of $Z_{\leq k-1}$ and we can assume that 
\begin{align}\label{eqn56x}
&\ W_{k, n}' \cap Z_{\leq k} = U_{k, n} \cap Z_{\leq k},\ &\ W_{k,n}' \cap {\mc M}_{I, \leq k-1} = W_{k-1,n}.
\end{align}
On the other hand, since $Z_k \setminus U_{k, n}$ is compact and the stratum ${\mc M}_{I, \lambda_k}$ is metrizable, one can find a sequence of shrinking neighborhoods 
\beqn
Z_k \setminus U_{k, n} \subset W_{k, n, m}'' \subset {\mc M}_{I, \lambda_k},\ m \geq 0.
\eeqn
Then define 
\beqn
W_{k, n}: = \bigcap_{l \leq n} (W_{k, l}' \cup W_{k, l, n-l}'' )
\eeqn
We claim that $W_{k, n}$ is a sequence of shrinking neighborhoods of $Z_{\leq k}$. By definition, $Z_{\leq k}$ is clearly contained in each $W_{k, n}$. Suppose $p \in W_{k, n}$ for all $n$, we need to show that $p \in Z_{\leq k}$. If $p \in {\mc M}_{I, \leq k-1}$, then by the second identity of \eqref{eqn56x}, we see $p \in W_{k-1, n}$ for all $n$, hence by the induction hypothesis, $p \in Z_{\leq k-1} \subset Z_{\leq k}$. Suppose $p \in {\mc M}_{I, \lambda_k}$. Then there exists a sufficiently large $l$ such that $p \notin W_{k, l}'$. Then for all $n$,
\beqn
p \in W_{k, n} \ \forall n\ \Longrightarrow p \in W_{k,l}' \cup W_{k, l, n-l}''\ \forall n\ \Longrightarrow p \in W_{k, l, n}''\ \forall n.
\eeqn
As $W_{k, l, n}''$ is a sequence of shrinking neighborhoods of $Z_k \setminus U_{k,l}$, it follows that $p\in Z_k \setminus U_{k, l} \subset Z_{\leq k}$. 
\end{proof}

Define 
\beqn
\Gammait_I:= \prod_{\alpha \in I} \Gammait_\alpha.
\eeqn

\begin{defn}(The $\Gammait_I$-action on the thickened moduli spaces)\label{defn59}
Given $\gamma = (\gamma_\alpha)_{\alpha \in I} \in \Gammait_I$ and a point $\tilde p \in {\mc M}_I$ represented by $({\mc C}, {\bm u}, ({\bm y}_\alpha), (\phi_\alpha), (e_\alpha))$, define 
\beqn
\gamma \cdot \tilde p \in {\mc M}_I
\eeqn
to be the point represented by the tuple 
\beqn
({\mc C}, {\bm u}, {\bm y}_\alpha, (\gamma_\alpha \circ \phi_\alpha), (\gamma_\alpha(e_\alpha)))
\eeqn
where we use the $\Gammait_\alpha$-action on the universal family ${\mc U}_\alpha$ and the $\Gammait_\alpha$-action on ${\bm E}_\alpha$. One can check that this gives an action of $\Gammait_I$ on ${\mc M}_I$. 
\end{defn}

Define 
\beqn
{\bm E}_I:= \bigoplus_{\alpha \in I} {\bm E}_\alpha.
\eeqn
There is a $\Gammait_I$-equivariant map 
\beqn
{\bm S}_I: {\mc M}_I \to {\bm E}_I
\eeqn
by forgetting all data except $(e_\alpha)_{\alpha \in I}$. One can see that each point in the zero locus ${\bm S}_I^{-1}(0)$ provides a gauge equivalence class of stable solutions to the gauged Witten equation with extra data $({\bm y}_\alpha)$ and $(\phi_\alpha)$, hence there is a well-defined $\Gammait_I$-invariant continuous map 
\beqn
{\bm \psi}_I: {\bm S}_I^{-1}(0) \to \ov{\mc M}_S(V, d).
\eeqn
Modulo the $\Gammait_I$-action we have a tuple 
\beq\label{eqn56}
C_I:= (U_I, E_I, S_I, \psi_I, F_I):= ( {\mc M}_I/ \Gammait_I, ({\mc M}_I \times {\bm E}_I)/\Gammait_I, {\bm S}_I/ \Gammait_I, {\bm \psi}_I / \Gammait_I, {\rm Im}( {\bm \psi}_I ))
\eeq
which will be the prototypical charts on the moduli space $\ov{\mc M}_S(V, d)$. The rest of the construction is to prove various properties of such charts, to carefully choose a covering collection, to construct coordinate changes and perturbations.

\subsection{Constructing basic charts}

We abbreviate $\ov{\mc M}_S(V, d)$ by $\ov{\mc M}$. We first prove that at each point $p \in \ov{\mc M}$ there exists a thickening datum.

\begin{prop}
For each $p\in \ov{\mc M}$ there exists a thickening datum $\alpha$ centered at $p$.
\end{prop}

\begin{proof}
Choose a representative $({\mc C}, {\bm u})$ for $p$ where ${\mc C}$ is an (unmarked) scaled curve and ${\bm u}$ is a stable pointlike instanton over ${\mc C}$. Let $\Gamma$ be the underlying combinatorial type of ${\mc C}$. 

We first describe the choice of stabilizing points ${\bm y}$ and the normalization hypersurface $H$. Consider each unstable vertex $v$ of $\Gamma$. If $v \in V_\Gamma^\infty$, then $u_v: \Sigma_v \cong {\mb C} \to X$ is a nontrivial holomorphic sphere. Then by the standard knowledge of pseudoholomorphic curves, there exists an open dense subset $\Sigma_v^* \subset \Sigma_v$ over which $u_v$ is an immersion. Then one can choose either one or two points over $\Sigma_v$ to stabilize this component. If we need to choose more than one such points, then we choose them such that their images are distinct. Denote by ${\bm y}_v$ this set of one or two points. Then one can choose for each $y \in {\bm y}_v$ an embedded submanifold (with boundary) $\ov{H}_y \subset X$  of codimension two such that $u_v$ intersects $\ov{H}_y$ transversely at $y$ (it is possibly that $u_v$ intersect $\ov{H}_y$ at other points. One can lift $\ov{H}_y$ to a $K$-invariant embedded codimension two submanifold $H_y \subset V$ which intersects $\mu^{-1}(0)$ transversely. Moreover, we replace each newly added marking by a new component together with a corresponding marking; the new component supports a covariantly constant pointlike instanton (see Subsection \ref{subsection52} and Figure \ref{figure1}).

On the other hand, if $v \in V_\Gamma^1$ and is unstable, then by Lemma \ref{lemmaa7}, there exists a nonempty open subset $\Sigma_v^* \subset \Sigma_v \cong {\mb C}$ over which the vectors $\partial_s u_v + {\mc X}_{\phi_v}(u_v)$ and $\partial_t u_v + {\mc X}_{\psi_v}(u_v)$ are linearly independent and over which the image of $u_v$ are contained in the neighborhood of $\mu^{-1}(0)$ where the $K$-action is free (i.e. a $K$-immersive point). Choose a point $y \in \Sigma_v^*$. Modify the gauged map $u_v = (u_v, \phi_v, \psi_v)$ via a suitable gauge transformation such that $u_v$ is an immersion near $y$.\footnote{This is necessary to deal with the case when the component $u_v$ lies entirely in the closure of a $G$-orbit.} Then choose an embedded submanifold  $H_y \subset V$ of codimension two (not necessarily $K$-invariant) such that $u_v$ intersects $H_y$ transversely at $y$. 

Denote by ${\bm y}$ the collection of all such points $y$ and ${\bm H}\subset V$ the union of all chosen embedded submanifolds $H_y$. We can choose them in an $\Gammait$-equivariant way. Indeed, by adding more points to ${\bm y}$ and adding more components to ${\bm H}$, ${\bm y}$ can be made $\Gammait$-invariant. Moreover, for each $\gamma \in \Gammait$ which contains a domain isomorphism 
\beqn
\gamma: \Sigma_v \to \Sigma_{v'}
\eeqn
mapping $y \in \Sigma_v \cap {\bm y}$ to $y' \in \Sigma_{v'} \cap {\bm y}$ and a gauge transformation $g_\gamma: \Sigma_v \to K$, there holds
\beqn
H_{y'} = g_\gamma(y) H_y.
\eeqn

We now describe the choice of the obstruction spaces. For each vertex $v\in V_\Gamma^\infty$, let ${\mc B}_v$ be the Banach manifold of $W^{1,p}$-maps from $S^2\cong \ov{\Sigma_v}$ to $X$ and let ${\mc E}_v \to {\mc B}_v$ be the Banach space bundle of $L^p$-sections of $\Lambda^{0,1} \otimes u_v^* TX$. Then the linearization of the Cauchy--Riemann equation at $u_v: S^2 \to X$ provides a linear Fredholm map
\beqn
D_v: W^{1,p}(S^2, u_v^* TX) \to L^p(S^2, \Lambda^{0,1}\otimes u_v^* TX).
\eeqn
Let $\mathring{W}^{1,p}(S^2, u_v^* TX) \subset W^{1,p}(S^2, u_v^* TX)$ be the subspace of infinitesimal deformations whose values at the special points (including those added marked points) are zero, which has a finite codimension. Then choose a finite-dimensional subspace $E_v \subset {\mc E}_v:=L^p(S^2, \Lambda^{0,1}\otimes u_v^* TX)$ generated by smooth sections with compact supports away from special points which is transverse to the image of the restriction of $D_v$ to $\mathring{W}^{1,p}(S^2, u_v^* TX)$. On the other hand, for each $v \in V_\Gamma^1$, Lemma \ref{lemmaa8} implies that there exists a finite-dimensional $E_v$ satisfies similar properties. Define 
\beqn
{\bm E}_\alpha:= \bigoplus_{v\in V_\Gamma} E_v \subset \bigoplus_{v\in V_\Gamma} {\mc E}_v.
\eeqn
We can choose $E_v$'s such that ${\bm E}_\alpha$ is a $\Gammait$-invariant subspace of the direct sum of ${\mc E}_v$'s. 

Lastly we choose an equivariant map $\iota_\alpha$. First, for each $v \in V_\Gamma^\infty$, one can use a gauged map ${\bm u}_v = (u_v, \phi_v, \psi_v)$ to represent the holomorphic map $u_v$. Then the image of $u_v$ is contained in the semi-stable locus $V^{\rm ss}$ where the $K$-action is free. Hence one can define an inclusion
\beqn
\iota_v: E_v \hookrightarrow \Gamma( \Sigma_v \times V, \Lambda^{0,1}\otimes TV)^K
\eeqn
whose support is contained in the product of the complement of special points and $V^{\rm ss}$, such that 
\beqn
u_v^* \circ \iota_v  = {\rm Id}_{E_v}.
\eeqn
Similarly, for $v\in V_\Gamma^1$, as the image of supports of vectors in $E_v$ is contained in $V^{\rm ss}$, one can embed $E_v$ into $\Gamma(\Sigma_v \times V, \Lambda^{0,1}\otimes TV)^K$ satisfying similar conditions. Lastly, using the trivialization \eqref{eqn55} contained in the resolution datum ${\bm r}_\alpha$, one can extend the embeddings $\iota_v$ to nearby fibres of the universal family ${\mc U}_\alpha \to {\mc V}_\alpha$ in a $\Gammait_\alpha$-equivariant way. This produces the linear map $\iota_\alpha$ satisfying conditions required in Definition \ref{defn54}.
\end{proof}

Now we prove properties of thickened moduli spaces corresponding to a single thickening datum. Recall that for a finite collection $I$ of thickening data, one has the thickened moduli space ${\mc M}_I$ with a $\Gammait_I$-action, an equivariant map ${\bm S}_I: {\mc M}_I \to {\bm E}_I$ and an invariant map ${\bm S}_I^{-1}(0) \to \ov{\mc M}$. This produces a tuple 
\beqn
C_I = (U_I, E_I, S_I, \psi_I, F_I)
\eeqn
(see \eqref{eqn56}). When $I = \{\alpha\}$, also denote the thickened moduli space by ${\mc M}_\alpha$ and the tuple by $C_\alpha$. After appropriate shrinking $C_\alpha$ becomes a chart of $\ov{\mc M}$.

\begin{lemma}\label{lemma511}
Let $\alpha$ be a thickening datum at a point $p \in \ov{\mc M}$. The following facts hold.
\begin{enumerate}

\item $U_\alpha$ is a locally compact, Hausdorff, and first countable. 

\item The image of map $\psi_\alpha: S_\alpha^{-1}(0) \to \ov{\mc M}$ contains an open neighborhood of $p$ and the restriction of $\psi_\alpha$ to a small neighborhood of the center is a homeomorphism onto an open neighborhood $F_\alpha \subset \ov{\mc M}$ of $p$. 

\end{enumerate}
\end{lemma}

\begin{proof}
Proof of (a). By using the same argument as in \cite[Section 5.6]{McDuff_Salamon_2004} (see also \cite[Remark 2.2]{Venugopalan_Xu}) the thickened moduli space has a unique first countable and separable topology whose set of converging sequences agrees with the set of Gromov--Uhlenbeck converging sequences. It is easy to see that $U_\alpha$ is locally compact.

Proof of (b). We argue by contradiction. Suppose it is not the case as claimed. Then there exists a sequence $p_i \in \ov{\mc M}$ converging to $p$ such that none of the $p_i$'s is contained in $F_\alpha$. Then by the definition of sequential convergence of stable pointlike instantons, one can choose a sequence of representatives $({\mc C}_i, {\bm u}_i)$ which Gromov converge to the representative $({\mc C}_\alpha, {\bm u}_\alpha)$. For notational simplicity we only prove the situation when the domain ${\mc C}_i$ is smooth; general situation is similar. Then we need to show that for $i$ sufficiently large, up to gauge transformation, $({\mc C}_i, {\bm u}_i)$ belongs to an $\alpha$-thickened solution. We first add stabilizing markings to ${\mc C}_i$. By the definition of convergence, for each $v \in V_\Gamma^1$, there is a sequence of affine linear transformations $\varphi_{i, v}: \Sigma_v \to {\mc C}_i \cong {\mb C}$ such that $\varphi_{i, v}^* {\bm u}_i$ converges modulo gauge transformation to ${\bm u}_{\alpha,v}$. For each marking $y \in \Sigma_v$, for sufficiently small neighborhood $B_r(y) \subset \Sigma_{\alpha, v} \subset {\mc C}_\alpha$, one can gauge transform the restriction ${\bm u}_i|_{\varphi_{i, v}(B_r(y))}$ such that $u_i(\varphi_{i,v}(B_r(y)))$ is contained in the standard $K$-slice through $u_\alpha(y)$. Then $u_i\circ \varphi_{i,v}|_{B_r(y)}$ converges in $C^\infty$ to $u_\alpha|_{B_r(y)}$. Since $u_\alpha$ is transverse to $H_{\alpha, y}$ at $y$, for $r$ sufficiently small and $i$ sufficiently large, there is a unique point $y_i \in \varphi_{i,v}(B_r(y))$ such that $u_i$ intersects transversely with $H_{\alpha, y}$ at $y_i$. Similarly, for each $v \in V_\Gamma^\infty$, there is a sequence of affine linear transformations $\varphi_{i,v}: \Sigma_v \to {\mc C}_i \cong {\mb C}$ such that ${\bm u}_i \circ \varphi_{i,v}$ converges in the large area adiabatic limit to $u_v$ (see Definition \ref{defn210}). Recall that for each stabilizing marking $y \in \Sigma_v$, $du_v(y): T_y \Sigma_v \to T_{u_v(y)} X$ is injective. Moreover, by Lemma \ref{lemma211}, $d(u_i \circ \varphi_{i,v})$ converges to $du_v$ modulo directions tangent to $G$-orbits uniformly on compact sets. Since $u_v$ intersects $H_{\alpha, y}$ transversely at $y$, it follows that $u_i \circ \varphi_{i,v}$ intersect to $H_{\alpha, y}$ at some point near $\varphi_{i,v}(y)$ and the intersection point $y_i$ is unique. Therefore one can also add $y_i$ to ${\mc C}_i$. Let ${\bm y}_i \subset {\mc C}$ be the collection of points selected in this way. Then it follows that the marked curves $({\mc C}_i, {\bm y}_i)$ converges to $({\mc C}_\alpha, {\bm y}_\alpha)$, or equivalently, for $i$ sufficiently large, there exist deformation parameters $\eta_i \to 0$ and gluing parameters $\zeta_i\to 0$ and an isomorphism 
\beqn
\phi_i: ({\mc C}_i, {\bm y}_i) \cong ({\mc C}_{\alpha, \eta_i, \zeta_i}, {\bm y}_{\alpha, \eta_i, \zeta_i}).
\eeqn
By our construction, the tuple $({\mc C}_i, {\bm u}_i, {\bm y}_i, \phi_i, 0)$ (where $0 \in {\bm E}_\alpha$ is the zero vector) is an $\alpha$-thickened solution. This is a contradiction. Hence $F_\alpha$ does contain an open neighborhood of $p$ in $\ov{\mc M}$. 

It remains to show that after restriction, $\psi_\alpha: S_\alpha^{-1}(0) \to F_\alpha$ is a homeomorphism. For this we show that $\psi_\alpha$ is injective. Suppose this is not the case, then there exist two $\alpha$-thickened solutions $({\mc C}, {\bm u}, {\bm y}, \phi, 0)$ and $({\mc C}', {\bm u}', {\bm y}', \phi', 0)$ such that $({\mc C}, {\bm u})$ is isomorphic to $({\mc C}', {\bm u}')$ as stable pointlike instantons. After identifying domains and gauge transforming, we may assume ${\mc C} = {\mc C}'$ and ${\bm u} = {\bm u}'$. We need to show that these two thickened solutions are in the same $\Gammait_\alpha$-orbit (see Definition \ref{defn59}). As ${\bm u}$ and ${\bm u}'$ are both (locally) close to ${\bm u}_\alpha$ and the intersections with ${\bm H}_\alpha$ are transverse, we see that as unordered sets, ${\bm y} = {\bm y}'$. By the construction of the local universal family, there exists a unique $\gamma \in {\rm Aut}({\mc C}_\alpha, {\bm y}_\alpha)$ such that
\beqn
\phi = \gamma \circ \phi': {\mc C} \to {\mc U}_\alpha.
\eeqn
Since the stabilizing markings ${\bm y}_\alpha$ are chosen such that $\Gammait_\alpha = {\rm Aut}({\mc C}_\alpha, {\bm u}_\alpha)  = {\rm Aut}({\mc C}_\alpha, {\bm y}_\alpha)$, it follows from Definition \ref{defn59} that the two thickened solutions are in the same $\Gammait_\alpha$-orbits. Therefore, $\psi_\alpha$ is injective. It needs an additional argument in point set topology to show that $\psi_\alpha$ is a homeomorphism. As $U_\alpha$ is locally compact, there is a precompact open neighborhood $U_\alpha'$ of the base point $\psi_\alpha^{-1}(p_\alpha)$. Since $\ov{\mc M}_\alpha$ is Hausdorff and the restriction $\psi_\alpha|_{\ov{U_\alpha'}}$ is one-to-one and continuous, $\psi_\alpha|_{\ov{U_\alpha'}}$ is then a homeomorphism onto its image.\footnote{Here we use the fact: a continuous bijective map from a compact space to a Hausdorff space is necessarily a homeomorphism.} Then by the surjectivity of $\psi_\alpha$ proved above, $\psi_\alpha|_{U_\alpha'}$ is a homeomorphism onto an open neighborhood $F_\alpha'$ of $p_\alpha$. We may just assume $U_\alpha$ itself satisfies this condition.
\end{proof}

Next we prove the so-called ``semi-continuity'' property, which implies that locally each stratum of the thickened moduli space is a topological manifold. 

\begin{lemma}\label{lemma512}
The thickened moduli space ${\mc M}_\alpha$ is naturally stratified by the map types and for each stratum $\lambda$, the stratum ${\mc M}_{\alpha, \lambda}$ is a topological manifold near $p$ of dimension ${\rm ind}(\lambda) + {\rm dim} {\bm E}_\alpha$. 
\end{lemma}

\begin{proof}
First, from the construction we know that there is a minimal stratum $\lambda_0$ for which ${\mc M}_{\alpha, \lambda_0} \neq \emptyset$. Then by the transversality assumption, ${\mc M}_{\alpha, \lambda_0}$ is a topological manifold of the right dimension near $p$. To show that each stratum ${\mc M}_{\alpha, \lambda}$ is a topological manifold of the right dimension, it suffices to show that for any $\alpha$-thickened solution in the stratum $\lambda$ which is sufficiently close to $p$, the linearized operator is surjective modulo the obstruction spaces. Comparing with the lowest nonempty stratum $\lambda_0$, if the map type $\lambda$ does not resolve any node connecting an instanton component and a holomorphic sphere component, then the transversality at nearby solutions in ${\mc M}_{\alpha, \lambda}$ holds for the same reason as the case of stable holomorphic maps. Hence the problem is reduced to the case that $\lambda_0$ is a codimension two stratum where the only higher stratum nearby is the top stratum. Namely, $p$ is represented by a stable marked pointlike instanton with only one holomorphic sphere component and several instanton components. In this situation, we prove the claim via the argument by contradiction. Suppose on the contrary, there is a sequence of smooth (marked) pointlike instantons $u_i$ converging to $p$ and the linearized operators $D_i$ at $u_i$ are not transverse to the obstruction space. Then there exists a sequence $\eta_i\in L^{q, -\delta}$ which are in the $L^2$-orthogonal complement of the image of $D_i$ and the obstruction space $E_i \subset L^{p,\delta}\cap C_0^\infty$ (here $p$ and $q$ are conjugates, see notations from the appendix). Then $\eta_i$ are in the kernel of the adjoint operator $D_i^*$. Assuming that $\eta_i$ has $L^{q, -\delta}$-norm being $1$, then elliptic estimates shows that over each component of $p$, $\eta_i$ converges (subsequentially) to an element in the cokernel of the linearization. This contradicts the assumption the original solution is transverse to the obstruction space.  
\end{proof}

As ${\mc M}_\alpha$ is locally compact and Hausdorff, a neighborhood ${\bm U}_\alpha \subset {\mc M}_\alpha$ is normal. Then it follows from Lemma \ref{lemma511} and  that the tuple $C_\alpha = (U_\alpha, E_\alpha, S_\alpha, \psi_\alpha, F_\alpha)$ is an abstract chart of $\ov{\mc M}$ (after proper shrinking) whose strata are topological orbifolds of the right dimensions.

\subsection{Constructing an atlas}

Fix a degree $d \in H_2^K(V; {\mb Z})$. Then by compactness, there are only finitely many map types $\lambda$ such that 
\beqn
\emptyset \neq {\mc M}_S(V, d)_\lambda \subset \ov{\mc M}_S(V,d).
\eeqn
One orders these types {\it increasingly}, i.e., orders them as $\lambda_1, \ldots, \lambda_m$ such that 
\beqn
\lambda_k \leq \lambda_l \Longrightarrow k \leq l.
\eeqn
For $l = 1, \ldots, m$, define 
\beqn
\ov{\mc M}{}^{(l)}:= \bigcup_{\alpha \leq l} {\mc M}_S(V, d)_{\lambda_\alpha}.
\eeqn
Then $\ov{\mc M}{}^{(l)}$ is compact. 

Now we state the main induction hypothesis for the construction of an atlas on $\ov{\mc M}$.

\begin{center}

{\sc Induction Hypothesis}

\end{center}

Given $k \in \{1, \ldots, m\}$. There exist the following objects.

\begin{enumerate}

\item A set of points $p_1, \ldots, p_{n_k} \in \ov{\mc M}{}^{(k)}$ together with a choice of thickening data and a corresponding stratified charts
\beqn
C_{p_i}^k = (U_{p_i}^k, E_{p_i}^k, S_{p_i}^k, \psi_{p_i}^k, F_{p_i}^k).
\eeqn
For each nonempty subset $I \subset \{ p_1, \ldots, p_{n_k}\}$, define 
\beqn
F_I^k:= \bigcap_{p_i \in I} F_{p_i}^k.
\eeqn

\item A collection of stratified charts $C_I^k = (U_I^k, E_I^k, S_I^k, \psi_I^k, F_I^k)$ for all $I \in{\mc I}_k$ where
\beq\label{eqn58}
{\mc I}_k:= \left\{ \emptyset \neq I = \{p_{i_1}, \ldots, p_{i_s}\}\subset \{p_1, \ldots, p_{n_k}\} \ |\ F_I^k \neq \emptyset \right\}.
\eeq

\item Define a partial order $\preceq$ on ${\mc I}_I$ by inclusion. Then for each pair $I \preceq J$, there is a stratified coordinate change
\beqn
T_{JI}^k = (U_{JI}^k, \hat \phi_{JI}^k)
\eeqn
from $C_I^k$ to $C_J^k$.

\end{enumerate}

They must satisfy the following conditions. 
\begin{enumerate}
\item When $I = \{p_i\}$, $C_{\{p_i\}}^k = C_{p_i}^k$ and $p_i \in F_{p_i}^k$.

\item $\ov{\mc M}{}^{(k)}$ is contained in the union of $F_{p_i}^k$. 

\item For each $I \in {\mc I}_k$, the chart $C_I^k$ is the $\Gammait_I$-quotient of an open subset ${\bm U}_I^k$ of the $I$-thickened moduli space ${\mc M}_I$

\item The collection of coordinate changes $T_{JI}^k$ satisfy the identity and cocycle conditions. 
\end{enumerate}

Here end the induction hypothesis.

We can see that if the induction hypothesis is satisfied for $k = m$ (the number of all involved strata), then we obtain an atlas of $\ov{\mc M}$. Now we start the induction. Notice that the base case of the induction can be recovered from the induction step. Suppose for $k \leq m-1$, the induction hypothesis holds. 

\vspace{0.2cm}

\noindent \underline{\it Step One.} Constructing the basic charts.

\vspace{0.2cm}

First we need to find additional basic charts to cover the $(k+1)$-st stratum ${\mc M}_{\lambda_{k+1}}$. One can find finitely many points 
\beqn
p_{n_k+1}, \ldots, p_{n_{k+1}} \in {\mc M}_{\lambda_{k+1}}
\eeqn
together with thickening data $\alpha_{n_k+1},\ldots, \alpha_{n_{k+1}}$ and corresponding basic charts $C_{p_{n_k + 1}}$, $\ldots$, $C_{p_{n_{k+1}}}$ with footprints $F_{p_{n_k + 1}}$, $\ldots$, $F_{p_{n_{k+1}}}$ such that
\beqn
\ov{\mc M}{}^{(k+1)} = \ov{\mc M}{}^{(k)} \cup {\mc M}_{\lambda_{k+1}} \subset \bigcup_{1 \leq i \leq n_k} F_{p_i}^k \cup \bigcup_{n_k + 1 \leq i \leq n_{k+1}} F_{p_i}.
\eeqn
Then choose precompact open subsets $F_{p_i}^k \sqsubset F_{p_i}$ for all newly added $p_i$ such that 
\beqn
\ov{\mc M}{}^{(k+1)} \subset \bigcup_{1 \leq i \leq n_{k+1}} F_{p_i}^k.
\eeqn
Then one has the partially ordered set ${\mc I}_{k+1}$ defined by \eqref{eqn58} which contains ${\mc I}_k$. For each $I \in {\mc I}_{k+1}$, define
\beqn
F_I^k:= \bigcap_{p_i \in I} F_{p_i}^k \subset \ov{\mc M}.
\eeqn

\vspace{0.2cm}

\noindent \underline{\it Step Two.} Constructing the sum charts.

\vspace{0.2cm}

The set of points $\{p_1, \ldots, p_{n_{k+1}}\}$ induces the set ${\mc I}_{k+1}$ by the formula \eqref{eqn58}. For each $I \in {\mc I}_{k+1}$, we want to construct a ``sum chart'' from the thickened moduli space. Suppose $I = \{p_{i_1}, \ldots, p_{i_s}, q_1, \ldots, q_m\}$ such that
\beqn
I \cap {\mc M}_{\lambda_{k+1}} =  \{q_1,\ldots, q_m\}.
\eeqn
Then we consider the $I$-thickened moduli space ${\mc M}_I$ where by abuse of notation, $I$ also denotes the set of thickening data we chose for the involved points. By previous discussion, there is a $\Gammait_I$-action on ${\mc M}_I$, an equivariant maps 
\beqn
{\bm S}_I: {\mc M}_I \to {\bm E}_I
\eeqn
and a $\Gammait_I$-invariant map
\beqn
{\bm \psi}_I: {\bm S}_I^{-1}(0) \to \ov{\mc M}.
\eeqn
This produces a tuple $C_I = (U_I, E_I, S_I, \psi_I, F_I)$ which would become a chart on $\ov{\mc M}$ after additional ``trimming.''

\begin{lemma}
There exists an $\Gammait_I$-invariant open neighborhood ${\bm U}_I\subset {\mc M}_I$ of ${\bm \psi}_I^{-1}(\ov{F_I^k})$ such that the induced map 
\beqn
\psi_I: ({\bm U}_I \cap {\bm S}_I^{-1}(0))/ \Gammait_I \to \ov{\mc M}
\eeqn
is a homeomorphism onto an open neighborhood of $\ov{F_I^{k}}$ and such that each stratum ${\bm U}_{I, \lambda}$ is a topological manifold. 
\end{lemma}

\begin{proof}
The proof is similar to that of Lemma \ref{lemma511} and Lemma \ref{lemma512}. We left the details to the reader. 
\end{proof}

Therefore, one can take a smaller $\Gammait_I$-invariant subset ${\bm U}_I^k \subset {\bm U}_I$ such that ${\bm S}_I^{-1}(0) \cap {\bm U}_I^k = {\bm \psi}_I^{-1}( F_I^k)$. This provides a chart 
\beqn
C_I^k = (U_I^k, E_I^k, S_I^k, \psi_I^k, F_I^k)
\eeqn
which is the $\Gammait_I$-quotient of $({\bm U}_I^k, {\bm E}_I^k, {\bm S}_I^k)$.

\vspace{0.2cm}

\noindent \underline{\it Step Three.} Constructing coordinate changes.

\vspace{0.2cm}

To construct coordinate changes, we first do a further shrinking on the basic charts. Consider an arbitrary collection of precompact open subset
\beqn
F_{p_i}^\bullet \sqsubset F_{p_i}^k,\ 1 \leq i \leq n_{k+1}
\eeqn
such that the union of $F_{p_i}^\bullet$ still cover $\ov{\mc M}{}^{(k+1)}$.\footnote{This is possible because each stratum of the moduli space is metrizable. See Remark \ref{rem57}.} Then
\beqn
F_I^\bullet:= \bigcap_{p_i \in I} F_{p_i}^\bullet \subset \bigcap_{p_i \in I} \ov{F_{p_i}^\bullet} \subset F_I^k.
\eeqn
For each pair $I, J \in {\mc I}_{k+1}$ with $I \subset J$ and $J \notin {\mc I}_k$, define
\begin{align*}
&\ {\bm E}_{JI}:= \bigoplus_{p_i\in J \setminus I} {\bm E}_{p_i},\ &\ \Gammait_{JI}:= \prod_{p_i \in J \setminus I} \Gammait_{p_i}.
\end{align*}
Define 
\beqn
{\bm S}_{JI}: {\mc M}_J \to {\bm E}_{JI}
\eeqn
to be the ${\bm E}_{JI}$-part of ${\bm S}_J$. Then there is a natural map 
\beqn
{\bm \psi}_{JI}: {\bm S}_{JI}^{-1}(0) \to {\mc M}_I
\eeqn
by forgetting ${\bm y}_\alpha, \phi_\alpha, e_\alpha$ for all $\alpha \in J \setminus I$ from the data. This is clearly equivariant with respect to the natural homomorphism $\Gammait_J \to \Gammait_I$ which annihilates $\Gammait_{JI}$, hence descends to a map 
\beqn
\psi_{JI}: {\bm S}_{JI}^{-1}(0)/ \Gammait_{JI} \to {\mc M}_I.
\eeqn

\begin{lemma}\label{lemma514}
There exists a $\Gammait_J$-invariant open neighborhood ${\bm N}_{JI} \subset {\mc M}_J$ of ${\bm \psi}_J^{-1}(\ov{F_J^\bullet})$ such that the map 
\beqn
\psi_{JI}: ( {\bm N}_{JI} \cap {\bm S}_{JI}^{-1}(0)) / \Gammait_{JI} \to {\mc M}_I
\eeqn
is a homeomorphism onto a $\Gammait_I$-invariant open neighborhood ${\bm U}_{JI}$ of ${\bm \psi}_I^{-1}( \ov{F_J^\bullet})\subset {\mc M}_I$. 
\end{lemma}

\begin{proof}
It is similar to the proof of Lemma \ref{lemma511} that the image of the map ${\bm \psi}_{JI}$ contains an open neighborhood ${\bm U}_{JI} \subset {\bm U}_I$ of ${\bm \psi}_I^{-1}(\ov{F_J^\bullet})$. It is also similar that when restricting to a small neighborhood of ${\bm \psi}_J^{-1}( \ov{F_J^k})$, the quotient of ${\bm \psi}_{JI}$,
\beqn
\psi_{JI}: ({\bm N}_{JI} \cap {\bm S}_{JI}^{-1}(0))/ \Gammait_{JI} \to {\mc M}_I
\eeqn
is injective. Therefore, $\psi_{JI}$ is a homeomorphism.
\end{proof}

Now we formally declare what the charts and the coordinate changes on the $k+1$-st stage are. We first need to shrink the charts $C_I^k$. For all $I \in {\mc I}_k$, we can shrink $C_I^k$ to a chart $C_I^\bullet$ with footprint being $F_I^\bullet$. For all $J \in {\mc I}_{k+1} \setminus {\mc I}_k$, we can choose a $\Gammait_J$-invariant open subset
\beq\label{eqn59}
{\bm U}_J^\bullet \subset \bigcap_{I \preceq J} {\bm N}_{JI} \subset {\mc M}_J
\eeq
(where ${\bm N}_{JI}$ is granted by Lemma \ref{lemma514}) such that 
\beqn
{\bm U}_J^\bullet\cap {\bm S}_J^{-1}(0) = {\bm \psi}_J^{-1}(F_J^\bullet).
\eeqn

Now we define the coordinate changes. Consider a pair $I \preceq J$. If $J \in {\mc I}_k$ (implying $I \in {\mc I}_k$), define the coordinate change $T_{JI}^\bullet$ from $C_I^\bullet$ to $C_J^\bullet$ as the restriction of $T_{JI}^k$. If $J \in {\mc I}_{k+1}\setminus {\mc I}_k$, define 
\beqn
U_{JI}:= {\bm U}_{JI} /\Gammait_I \subset U_I^k
\eeqn
(where ${\bm U}_{JI}$ is given by Lemma \ref{lemma514}) and define
\beqn
\phi_{JI}: U_{JI} \to U_J^k
\eeqn
which is induced from the inverse of $\psi_{JI}$ given by Lemma \ref{lemma514} above. This is a topological embedding. Then define
\beq\label{eqn510}
U_{JI}^\bullet:= U_{JI} \cap U_I^\bullet \cap \phi_{JI}^{-1}(U_J^\bullet)
\eeq
which is the $\Gammait_I$-quotient of an open subset ${\bm U}_{JI}^\bullet \subset {\bm U}_I^\bullet$. Using the natural inclusion ${\bm E}_I \hookrightarrow {\bm E}_J$, one obtains a topological embedding 
\beqn
\hat\phi_{JI}^\bullet: ( {\bm U}_{JI}^\bullet \times {\bm E}_I)/\Gammait_I \to ({\bm U}_J^\bullet \times {\bm E}_J)/\Gammait_J.
\eeqn

\begin{lemma}
$T_{JI}^\bullet = (U_{JI}^\bullet, \hat\phi_{JI}^\bullet)$ is a coordinate change from $C_I^\bullet$ to $C_J^\bullet$.
\end{lemma}

\begin{proof}
Conditions (a), (b), and (c) of Definition \ref{defn42} are true automatically. To verify condition (d), take $x_n \in U_{JI}^\bullet$ and $y_n = \phi_{JI}^\bullet(x_n) \in U_J^\bullet$ and assume that $x_n \to x_\infty \in U_I^\bullet$, $y_n \to y_\infty \in U_J^\bullet$. Because $y_n$ is a $\Gammait_J$-orbit of $J$-thickened solutions contained in the locus ${\bm S}_{JI}^{-1}(0)$, $y_\infty$ is also contained in the locus ${\bm S}_{JI}^{-1}(0)$. Moreover, by \eqref{eqn59}, we see $y_\infty$ is contained in ${\bm N}_{JI}$. Then by Lemma \ref{lemma514}, we have $y_\infty \in \phi_{JI}(U_{JI})$. Therefore, by \eqref{eqn510}, we see $x_\infty \in U_{JI}^\bullet$ and hence $y_\infty = \phi_{JI}^\bullet( x_\infty)$. This proves condition (d) of Definition \ref{defn42}.
\end{proof}

Then the induction can be continuated if we replace all script $\bullet$ by the index $k+1$.

\subsection{Constructing a good coordinate system}\label{subsection56}

Suppose we have finished the last ($m$-th) step of the above induction. For the objects obtained, we remove the superscript $k = m$. In particular, we obtained an atlas ${\mc A}$ on $\ov{\mc M}$ whose charts are indexed by ${\mc I}$. To obtain a good coordinate system, we perform a typical shrinking argument as follows. 

Suppose we have chosen $n$ basic charts $C_{p_i}$, $i = 1, \ldots, n$. For each $p_i$, choose precompact shrinkings
\beq\label{eqn511}
\mathring{F}_{p_i;0} \sqsubset \mathring{G}_{p_i; 1} \sqsubset \mathring{F}_{p_i;1} \sqsubset \cdots \sqsubset \mathring{G}_{p_i; n} \sqsubset \mathring{F}_{p_i; n} \sqsubset F_{p_i}\footnote{There is a small difference from the counterpart in the proof of \cite[Lemma 5.3.1]{MW_3}, which took $F_{p_i; n_{k+1}}^{k} = F_{p_i}^{k}$. }
\eeq
such that the union of $\mathring{F}_{p_i; 0}$ covers $\ov{\mc M}$. Then define for all $I \in {\mc I}$
\beqn
\mathring{F}_I:= \left( \bigcap_{p_i \in I} \mathring{G}_{p_i; |I|} \right) \setminus \left( \bigcup_{p_i \notin I} \ov{\mathring{F}_{p_i; |I|}} \right).
\eeqn

\begin{lemma}\label{lemma516}
The collection $\{ \mathring{F}_I\ |\ I \in {\mc I} \}$ of open sets satisfies the following conditions. 

\begin{enumerate}

\item Its union covers $\ov{\mc M}$;

\item It satisfies the overlapping condition, i.e.
\beqn
\ov{\mathring{F}_I} \cap \ov{\mathring{F}_J} \neq \emptyset \Longrightarrow I \preceq J\ {\rm or}\ J \preceq I.
\eeqn
\end{enumerate}
\end{lemma}

\begin{proof}
See the proof of \cite[Lemma 5.3.1]{MW_3}. 
\end{proof}

Now we can construct a good coordinate system $\mathring{\mc A}$. By definition $\mathring{F}_I$ is a precompact open subset of $F_I$. Then one can find precompact open subsets 
\beqn
\mathring{U}_I \sqsubset U_I,\ \forall I \in {\mc I}
\eeqn
such that the induced shrinking $\mathring{C}_I$ of $C_I$ has footprint $\mathring{F}_I$. This gives us a collection of charts $\mathring{C}_I$ and a collection of induced coordinate changes $\mathring{T}_{JI}$ from $\mathring{C}_I$ to $\mathring{C}_J$ whenever $I \preceq J$. Then all conditions of Definition \ref{defn43} are satisfied. Hence $\mathring{\mc A}$ is an atlas. 

Then, by Lemma \ref{lemma521} one can shrink this atlas to a good coordinate system, which is still denoted by $\mathring{\mc A}$ such that each footprint is a slight precompact shrinking of $\mathring{F}_I$. This can be done by slightly enlarging $\mathring{F}_{p_i; l}$ and slightly shrinking $\mathring{G}_{p_i; l}$ while preserving \eqref{eqn511}. Moreover, as the charts are obtained from existing ones via precompact shrinkings, we may assume that the domains $\mathring{U}_I$ resp. $\mathring{U}_{JI}$ have natural compactifications $\ov{\mathring{U}_I}$ resp. $\ov{\mathring{U}_{JI}}$ such that the coordinate changes extend to continuous maps 
\beqn
\mathring{\phi}_{JI}: \ov{\mathring{U}_{JI}} \to \ov{\mathring{U}_J}.
\eeqn 
satisfying the cocycle condition. 

\subsection{Constructing perturbations}

Now we construct perturbations. We assume that we have obtained a good coordinate system ${\mc A}$  over $\ov{\mc M}$, whose charts $C_I = (U_I, E_I, S_I, \psi_I, F_I)$ are indexed by $I \in {\mc I}$. To construct perturbations, it is convenient to choose a shrinking ${\mc A}'$ of ${\mc A}$, induced by a precompact shrinking 
\beqn
U_I' \sqsubset U_I
\eeqn
for all $I \in {\mc I}$. We also need to make extra choices to help us construct suitable perturbations. 

\begin{enumerate}

\item Choose $\Gammait_I$-invariant metrics on the vector spaces ${\bm E}_I$ such that the inclusions ${\bm E}_I \to {\bm E}_J$ are all isometric. 

\item Choose, for each pair $I, J \in {\mc I}$ with $I \preceq J$, an open neighborhood $N_{JI}$ of $\phi_{JI}(U_{JI})$ such that

\begin{enumerate}
\item For any pair $I, J$ with $I \preceq J$, 
\beqn
S_{J \setminus I}^{-1}(0) \cap N_{JI} = \phi_{JI}(U_{JI})
\eeqn

\item For any triple $I, J, K$ such that $I \preceq K$, $J \preceq K$, and there is no partial order relation between $I$ and $J$, there holds $N_{KI} \cap N_{KJ} = \emptyset$. 

\end{enumerate}

\item Choose an open neighborhood $N_{JI}' \subset U_J'$ of $\phi_{JI}(U_{JI}')$ such that $\ov{N_{JI}'}\subset N_{JI}$ is compact. 
\end{enumerate}

Now we state how to construct transverse perturbations over the top stratum.

\begin{lemma}
For any sufficiently small $\epsilon>0$, there exist a collection of $\Gammait_I$-liftable multisections $t_I : U_I \overset{m}{\to} E_I$ for all $I \in {\mc I}$ satisfying the following conditions. 
\begin{enumerate}

\item For all $I \preceq J$, $t_J \circ \phi_{JI}$ and $\hat \phi_{JI} \circ t_I$ agrees in the strong sense over a neighborhood of $\ov{U_I'} \cap \phi_{JI}^{-1}(\ov{U_J'})$.

\item For all $I \preceq J$, over a neighborhood of $\ov{N_{JI}'}$ the value of $t_J$ is contained in the subbundle $E_I$. 

\item Define $\hat S_I:= S_I + t_I$. Then $\hat S_I$ is transverse over each stratum near $\ov{U_J'}$. 

\item For every $I$ one has $\| t_I \|_{C^0(\ov{U_I'})} < \epsilon$.

\end{enumerate}
\end{lemma}

\begin{proof}
The proof is very similar to the traditional case of constructing perturbations on a good coordinate system (see \cite[Theorem 7.37]{Tian_Xu_geometric}). We need to use the stratified version of the transversality extension theorem, which is Lemma \ref{lemma410}, in the induction procedure. We omit the details.
\end{proof}

We also need to control the compactness of the perturbed zero locus. Notice that the collection $t_I$ gives a multivalued perturbation $\hat S_I':= S_I' + t_I'$ on the good coordinate system ${\mc A}'$. Its zero locus $|(\hat S_I')^{-1}(0)|$ is a subset of the corresponding virtual neighborhood $|{\mc A}'|$. 

\begin{lemma}
For $\epsilon$ sufficiently small, $|(\hat S_I')^{-1}(0)|$ is sequentially compact in $|{\mc A}'|$. 
\end{lemma}

\begin{proof}
The argument is similar to the traditional case (see \cite[Proposition 7.38]{Tian_Xu_geometric}). The difference is that as we do not know if $U_I'$ is metrizable, we cannot use a distance function. However the condition that the $C^0$-norm of the perturbation can be arbitrarily small is enough to carry out the same argument.
\end{proof}

\subsection{Virtual count and invariance}

Now one can define the virtual count. For each $p \in |\hat{\mf s}^{-1}(0)|$, choose a representative $x_I \in U_I$ at which a branch of $\hat s_I$ vanishes. Suppose $\hat s_I$ can be represented by a $\Gammait_I$-liftable multimap with $k_I$ branches
\beqn
\hat {\bm s}_I = [\hat s_1, \ldots, \hat s_{k_I}]: {\bm U}_I \to S^{k_I}({\bm E}_I).
\eeqn
Suppose ${\bm x}_I \in {\bm U}_I$ is in the orbit $x_I$. Then define 
\beqn
\epsilon_p:= \frac{1}{k_I} \sum_{i=1}^{k_I} \epsilon_i
\eeqn
where $\epsilon_i = 0$ if $\hat s_i ({\bm x}_I) \neq 0$ and $\epsilon_i = \pm 1$ if $\hat s_i({\bm x}_I) = 0$ and the sign is determined in the following way. By our construction, we know that the perturbation only vanishes on the top stratum. For the chart $C_I = (U_I, E_I, S_I, \psi_I, F_I)$, we know that $U_I$ is the $\Gammait_I$-quotient of ${\bm U}_I\subset {\mc M}_I$. The tangent space of the top stratum of ${\bm U}_I$ at the point ${\bm x}_I$, which is represented by an $I$-thickened solution whose underlying gauged map is $(u, \phi, \psi)$, is isomorphic to the kernel of the map 
\beqn
\iota_I \oplus \hat D_u: E_I \oplus W^{1,p,\delta}({\mb C}, u^* TX \oplus {\mf k} \oplus {\mf k}) \to L^{p,\delta}({\mb C}, u^* TX\oplus {\mf k} \oplus {\mf k}).
\eeqn
Here $\hat D_u$ is the augmented linearized operator (see \eqref{aug}). As $J$ is integrable, we know that $\hat D_u$ is complex linear. Therefore the operator $\hat D_u$ carries a natural orientation. Therefore, over the manifold ${\bm U}_I$, the determinant line
\beqn
\det T{\bm U}_I \otimes \det {\bm E}_I^* \cong \det \hat D_u
\eeqn
has a natural trivialization which is $\Gammait_I$-equivariant. As $\hat s_i$ vanishes transversely at ${\bm x}_I$, the above orientation determines a sign $\epsilon_i$. This sign is clearly independent of the choice of the point ${\bm x}_I$ in its orbit and independent of the chart $C_I$. 

Lastly we define 
\beqn
\# |\hat {\mf s}^{-1}(0)| = \sum_{p \in |\hat {\mf s}^{-1}(0)|} \epsilon_p \in {\mb Q}.
\eeqn
We call this rational number the {\bf virtual count} of the moduli space $\ov{\mc M}_S(V, d)$.

\begin{rem}
We give a remark on how to prove the independent of this virtual count on various choices made in the construction. Given a good atlas ${\mc A}$ on $\ov{\mc M}_S(V, d)$ and a system of perturbations ${\mc T}$, one can define the virtual count or the virtual Euler number as the weighted count of zeroes. To prove that this number is independent of the construction, one use the typical cobordism argument. Consider the product $[0,1]\times \ov{\mc M}_S(V, d)$, which can be stratified by 
\beqn
\{0\}\times {\mc M}_\Gamma(V),\ \{1\}\times {\mc M}_\Gamma(V),\ (0,1)\times {\mc M}_\Gamma(V),\ \Gamma \in \Lambda_d.
\eeqn
The method of constructing sum charts can be extended to this product to produce a good coordinate system ${\mc A}_{[0,1]}$ together with a perturbation ${\mc T}_{[0,1]}$ whose restriction to the $\{0\}$-slice and the $\{1\}$-slice agree with the two given good coordinate systems together with perturbations. Then in the virtual neighborhood $|{\mc A}_{[0,1]}|$ the zero locus of ${\mc T}_{[0,1]}$ is a weighted 1-manifold providing an oriented cobordism between the zero locus of ${\mc T}_{0}$ and that of ${\mc T}_{1}$. It follows that the virtual counts defined for the two choices coincide. 

\end{rem}

\begin{rem}\label{rem520}
We also remark that one can extend our construction to the situation when $S$ is a pseudocycle rather than a compact submanifold. Indeed, if $S \subset X$ is a pseudocycle, then $\ov{S} \setminus S$ is the union of submanifolds $S_\beta$ of codimension at least two. Then one can consider a compact moduli space $\ov{\mc M}_{\ov{S}}(V, d) \subset \ov{\mc M}(V, d)$ which contains strata ${\mc M}_S(V, d)_\lambda$ or ${\mc M}_{S_\beta}(V, d)_\lambda$. Then a similar construction can be carried out.
\end{rem}

\subsection{Shrinking good coordinate systems}

We prove the following lemma which was used in induction ({\it Step Four} of Subsection \ref{subsection56}).

\begin{lemma}\label{lemma521}
Let $\ov{\mc M}$ be a compact Hausdorff space. Let ${\mc A} = ({\mc I}, (C_I), (T_{JI}))$ be a virtual atlas of ${\mc M}$. For each $I\in {\mc I}$, let $F_I^\circ \subset F_I$ be a precompact open subset of the footprint $F_I$ such that 
\beqn
\ov{\mc M} = \bigcup_{I \in {\mc I}} F_I^\circ.
\eeqn
Then there exist a precompact shrinking ${\mc A}' = ({\mc I}, (C_I'), (T_{JI}'))$ of ${\mc A}$ which is a good coordinate system such that $\ov{F_I^\circ} \subset F_I'$. 
\end{lemma}

\begin{proof}
In \cite{Tian_Xu_geometric} similar statement is proved for the stronger notion of {\it virtual orbifold atlas}. The proof of the current lemma is similar. We first find a shrinking ${\mc A}'$ such that the relation $\curlyvee'$ on ${\mc A}'$ is an equivalence relation. The argument of \cite[Lemma 7.30]{Tian_Xu_geometric} is to look for shrinkings for each triple $I, J, K$. The only property used in the proof is that for each chart $C_I = (U_I, E_I, S_I, \psi_I, F_I)$, any compact subset $K \subset U_I$ has a sequence of shrinking neighborhoods. This is precisely Corollary \ref{cor58}. Therefore one can assume that the relation $\curlyvee$ on ${\mc A}$ is already an equivalence relation. Next we need to show that after certain shrinking the virtual neighborhood $|{\mc A}'|$ is Hausdorff and the inclusions $U_I' \hookrightarrow |{\mc A}'|$ is a homeomorphism onto its image. In the setting of \cite{Tian_Xu_geometric}, this was proved via its Lemma 7.31, 7.32, and 7.33. In fact Lemma 7.32 and 7.33 carry over word by word. For \cite[Lemma 7.31]{Tian_Xu_geometric} notice the only property one uses is that each $U_I$ is locally compact, first countable, and Hausdorff. Therefore one can prove the current lemma in the same way as proving \cite[Lemma 7.31]{Tian_Xu_geometric}.
\end{proof}

\appendix

\section{Analysis of Pointlike Instantons}\label{appendixa}

In this appendix we prove several analytical ingredients in the construction. We first prove that pointlike instantons are also affine vortices. Then using known results or methods about the latter, we establish local model and Fredholm theory and prove the existence of finite-dimensional obstruction spaces.

\subsection{Pointlike instantons are affine vortices}

\begin{proof}[Proof of Proposition \ref{prop28}]
We only give a sketch as it is very similar to the proof of \cite[Theorem 5.1]{Tian_Xu_geometric}. First we can prove that bounded solutions satisfy the asymptotic condition 
\beqn
\lim_{z\to \infty} |\nabla W(u(z))| = \lim_{z \to \infty} |\mu(u(z))| = 0.
\eeqn
Define 
\begin{align*}
&\ u_s:= \partial_s u + {\mc X}_\phi(u),\ &\ u_t:= \partial_t u + {\mc X}_\psi(u).
\end{align*}
Then by the calculations and estimates in the proofs of \cite[Lemma 5.2, Lemma 5.3]{Tian_Xu_geometric}, one has 
\beqn
\Delta \left( |u_s|^2 + |u_t|^2 \right) \geq - C \left( 1 + (|u_s|^2 + |u_t|^2)^2 \right).
\eeqn
Let $\tau$, $\theta$ be the cylindrical coordinates defined via $s + {\bm i} t = e^{\tau + {\bm i} \theta}$. Then one has 
\beqn
\Delta_{\rm cyl} \left( |u_\theta|^2 + |u_\tau|^2 \right) \geq -C \left( 1 + ( |u_\theta|^2 + |u_\tau|^2)^2 \right).
\eeqn
By the mean value inequality one has 
\beqn
\lim_{|z|\to \infty} ( |u_\theta|^2 + |u_\tau|^2 ) = 0.
\eeqn
It follows from the first equation of \eqref{eqn21} that 
\beqn
\lim_{|z|\to \infty} |z| |\nabla W(u(z))| = 0.
\eeqn
As the solution is bounded and as $0$ is the only critical value of $W$ (\cite[Lemma 3.4]{Tian_Xu_geometric}), one has
\beq\label{eqna1}
\lim_{|z|\to \infty} |z| |W(u(z))| = 0.
\eeq
Then for all $R>0$, one has 
\begin{multline*}
\| u_s + J u_t\|_{L^2({\mb C})}^2 + \|\nabla W(u)\|_{L^2({\mb C})}^2  = \| u_s + J u_t + \nabla W(u)\|_{L^2({\mb C})}^2 - \int_{\mb C} \langle u_s + J u_t, \nabla W(u) \rangle ds dt \\
= - {\bm i} \int_{{\mb C}} dW(u) \cdot (u_s + J u_t) = {\bm i} \int_{\mb C} \ov\partial (W(u) dz)= {\bm i} \int_{\mb C} d (W(u) dz) = {\bm i}\lim_{R \to \infty} \int_{\partial B_R} W(u) dz.
\end{multline*}
By \eqref{eqna1} the right hand side is zero. Hence $u_s + Ju_t \equiv \nabla W(u) \equiv 0$.
\end{proof}

\subsection{Local model}

The theorem about local model for pointlike instantons can be derived with the aid of analogous results about affine vortices obtained in \cite{Venugopalan_Xu}. We recall the main theorem of \cite{Venugopalan_Xu} here. Let $u = (u, \phi, \psi)$ be an affine vortex in the K\"ahler manifold $(V, \omega, J)$ acted by the compact Lie group $K$ with a moment map $\mu: V \to {\mf k}^*$. Let $A$ be the connection $A = d + \phi ds + \psi dt$ on the trivial bundle $K \times {\mb C} \to {\mb C}$. Assume that $0$ is a regular value of $\mu$ and $K$ acts freely on $\mu^{-1}(0)$. Choose $p>2$ and $\delta \in (1 - \frac{2}{p}, 1)$. Define the norm on triples $\xi = (\xi_0, \xi_1, \xi_2)\in C_0^\infty( {\mb C}, u^* TV \oplus {\mf k} \oplus {\mf k})$ as 
\beqn
\| \xi \|_{1,p,\delta} = \| \xi_0 \|_{L^\infty} + \| \xi_1 \|_{L^{p,\delta}} + \| \xi_2\|_{L^{p,\delta}} +  \| \nabla^A \xi \|_{L^{p,\delta}} + \| d\mu(u)(\xi) \|_{L^{p,\delta}} + \| d\mu(u) (J \xi ) \|_{L^{p,\delta}}.
\eeqn
Let $\tilde {\mc B}_u$ be the completion of $C_0^\infty({\mb C}, u^* TV \oplus {\mf k} \oplus {\mf k})$ with respect to this norm which consists of $W^{1,p}_{\rm loc}$ sections of $u^* TV \oplus {\mf k} \oplus {\mf k}$ with finite norm. It is shown in \cite[Theorem 4(i)]{Ziltener_book} that $\tilde {\mc B}_u$ is a complete Banach space. The norm is obviously gauge invariant. Namely, if $g: {\mb C} \to K$ is a smooth gauge transformation and denote $u' = (u', \phi', \psi') = g^* u$, $A' = d + \phi' ds + \psi' dt$. Then for any $\xi = (\xi_0, \xi_1, \xi_2)$, denoting $\xi' = (\xi_0', \xi_1', \xi_2'):= g^* \xi$, there holds
\beqn
\| \xi  \|_{1,p,\delta} = \| \xi' \|_{1,p,\delta}
\eeqn
where the left hand side is defined with respect to $u$ and the right hand side is defined with respect to $u'$. Also notice that by the Sobolev embedding $W^{1,p}\hookrightarrow C^0$ in two dimensions and the fact that $\delta>0$, one has the estimate
\beqn
\| \xi \|_{L^\infty} \leq C \| \xi \|_{1,p,\delta},\ \forall \xi \in \tilde {\mc B}_u.
\eeqn

We call $\tilde {\mc B}_u$ the space of infinitesimal deformations of $u$ (with certain Sobolev regularity). We use an exponential map to identify small infinitesimal deformations with nearby gauged maps. For this we choose a particular kind of $K$-invariant Riemannian metric on $V$, i.e., a $K$-invariant metric such that $\mu^{-1}(0)$ is a totally geodesic submanifold. Let $\exp$ be its exponential map. 

Now for $\epsilon>0$ sufficiently small, one defines for $\xi \in \tilde {\mc B}_u^\epsilon$ a gauged map
\beqn
u_\xi:= (u_\xi, \phi_\xi, \psi_\xi):= (\exp_u \xi_0, \phi + \xi_1, \psi + \xi_2).
\eeqn
Also define for each $\xi \in \tilde {\mc B}_u^\epsilon$ a Sobolev space
\beqn
{\mc E}_{u_\xi}:= L^{p,\delta}(u_\xi^* TV \oplus {\mf k}).
\eeqn
Then define the nonlinear map 
\beqn
\begin{array}{rcc} 
\tilde {\mc F}_u: \tilde {\mc B}_u^\epsilon & \to & \displaystyle \bigsqcup_{\xi \in \tilde {\mc B}_u^\epsilon} {\mc E}_{u_\xi}\\
                  \xi & \mapsto & \left(\begin{array}{c} \partial_s u_\xi + {\mc X}_{\phi_\xi} (u_\xi) + J( \partial_t u_\xi + {\mc X}_{\psi_\xi}(u_\xi))\\ \partial_s \psi_\xi - \partial_t \phi_\xi + [\phi_\xi, \psi_\xi] + \mu(u_\xi)\end{array}\right)                 .
\end{array}
\eeqn
Then each zero $\xi \in \tilde {\mc F}_u^{-1}(0)$ gives an affine vortex $u_\xi$. The linearization of $\tilde {\mc F}_u$ at $\xi = 0$ is a first-order differential operator
\beqn
\tilde D_u(\xi) = \left( \begin{array}{c} \partial_s \xi_0 + \nabla_{\xi_0} {\mc X}_\phi + J( \nabla_t \xi_0 + \nabla_{\xi_0} {\mc X}_{\psi}) + {\mc X}_{\xi_1} + J {\mc X}_{\xi_2} \\
                                          \partial_s \xi_2 + [\phi, \xi_2] - \partial_t \xi_1 - [\psi, \xi_1] + d\mu(u) (\xi_0)\end{array}\right).
\eeqn

Define the so-called ``Coulomb slice''
\beqn
{\mc B}_u:= \big\{ \xi \in \tilde {\mc B}_u\ |\ \partial_s \xi_1 + [\phi, \xi_1] + \partial_t \xi_2 + [\psi, \xi_2] + d\mu(u)(J\xi_0) = 0 \big\}.
\eeqn
Denote ${\mc B}_u^\epsilon:= {\mc B}_u \cap \tilde {\mc B}_u^\epsilon$, ${\mc F}_u:= \tilde {\mc F}_u|_{{\mc B}_u^\epsilon}$, and $D_u:= \tilde D_u|_{{\mc B}_u}$. Then one has a natural map 
\beq\label{eqna2}
\varphi_u^\epsilon: {\mc F}_u^{-1}(0) \to \tilde {\mc N} (V),\ \xi \mapsto [u_\xi].
\eeq
Notice that $\tilde {\mc N}(V)$ is the set of gauge equivalence classes of bounded solutions to the vortex equation without quotienting out the translation symmetry.

\begin{prop}\cite{Ziltener_book}
The linear map $D_u: {\mc B}_u \to {\mc E}_u$ is Fredholm.
\end{prop}

We call an affine vortex $u$ {\bf regular} if $D_u$ is surjective.

\begin{thm}\cite[Proposition 3.3, 3.4]{Venugopalan_Xu}\label{thma2} Suppose $u$ is a regular affine vortex. Then there exists $\epsilon>0$ such that the natural map $\varphi_u^\epsilon: {\mc F}_u^{-1}(0) \to \tilde {\mc N} (V)$ is a homeomorphism onto an open neighborhood of $[u]$. 
\end{thm}

\begin{rem}
The regularity assumption on $u$ is unnecessary. When $u$ is not regular, the Fredholm property allows us to choose a finite-dimensional obstruction space $E$ generated by sections $s_1, \ldots, s_k \in C_0^\infty({\mb C}, u^*T V) \subset L^{p,\delta}({\mb C}, u^* TV \oplus {\mf k})$ such that the image of $D_u$ is transverse to this obstruction space (see Lemma \ref{lemmaa8} below). Then one can define the thickened moduli space $\tilde {\mc N}_E(V)$ consisting of gauge equivalence classes of gauged maps $ u = (u, \phi, \psi)$ from ${\mb C}$ to $V$ together with linear combinations $a_1 s_1 + \cdots + a_k s_k$ such that $(u,\phi, \psi)$ solves the vortex equation modulo $a_1 s_1 + \cdots + a_k s_k$. Similar argument as proving Theorem \ref{thma2} identifies the thickened moduli space with a zero locus of a nonlinear Fredholm map which, by transversality, is a manifold near $[u]$. 
\end{rem}

Theorem \ref{thma2} also provides a local coordinate chart for a Banach manifold which is difficult to describe explicitly. Indeed, for each element $[u] \in \tilde {\mc N}(V)$ and each representative $(u, \phi, \psi)$, we have described an open ball of a Banach space ${\mc B}_u$, another Banach space ${\mc E}_u$, and a Fredholm section ${\mc F}_u: {\mc B}_u \to {\mc E}_u$ whose zero locus is identical to a neighborhood of $[u]$ in ${\mc N}(V)$. Now suppose $[u], [u']\in \tilde {\mc N}(V)$ have representatives $(u, \phi,\psi)$ and $(u', \phi', \psi')$. Then by \cite[Proposition 3.8]{Venugopalan_Xu}, whenever ${\mc B}_u \cap {\mc B}_{u'} \neq \emptyset$ (as sets of gauge equivalence classes of gauged maps), there is a smooth coordinate changes. The cocycle condition among coordinate changes are obvious from their definition. Hence one obtains a smooth Banach manifold ${\mc B}$. At the same time the Banach spaces ${\mc E}_u$ glue together to a Banach vector bundle ${\mc E} \to {\mc B}$ and the maps ${\mc F}_u$ glue together to a section ${\mc F}: {\mc B} \to {\mc E}$. Then one has the homeomorphism 
\beqn
{\mc F}^{-1}(0) \cong \tilde {\mc N}(V).
\eeqn

\subsubsection{The case of the gauged Witten equation}

Now let $u= (u, \phi, \psi)$ be a pointlike instanton, i.e., a finite energy solution to the gauged Witten equation over ${\mb C}$. By Proposition \ref{prop28} $u$ is also an affine vortex. Now define a different deformation space
\beqn
\tilde {\mc W}_u: = \{ \xi \in \tilde {\mc B}_u\ |\ \| \nabla_{\xi_0} \nabla W(u) \|_{L^{p,\delta}} < \infty\}
\eeqn
with the stronger norm
\beqn
\| \xi \|_{1,p,\delta}^W:= \| \xi \|_{1,p,\delta} + \| \nabla_{\xi_0} \nabla W(u) \|_{L^{p,\delta}}.
\eeqn
Define the linear map 
\beqn
\begin{array}{rcc}
\tilde D_u^W: \tilde {\mc W}_u &\ \to &\ {\mc E}_u, \\
                 \xi & \mapsto & \tilde D_u(\xi) + \left( \begin{array}{c} \nabla_{\xi_0} \nabla W(u) \\ 0 \end{array}\right).
\end{array}
\eeqn
Define the Coulomb slice ${\mc W}_u:= \tilde {\mc W}_u \cap {\mc B}_u$ and denote $D_u^W:= \tilde D_u^W |_{{\mc W}_u}$. 

\begin{prop}\label{propa4}
$D_u^W$ is Fredholm with 
\beq\label{eqna3}
{\rm ind} D_u^W = m_d:= {\rm dim}_{\mb R} X + 2\langle c_1^K(TV), d \rangle 
\eeq
\end{prop}

\begin{proof}
It is similar to the computation of Fredholm index in \cite{Ziltener_book} and \cite{Venugopalan_Xu}. We first introduce the so-called augmented linearization, which reads 
\beq\label{aug}
\begin{split}
\hat D_u^W: \tilde {\mc B}_u \to & \tilde {\mc E}_u:= L^{p,\delta}({\mb C}, u^* TV \oplus {\mf k} \oplus {\mf k})\\
(\xi_0, \xi_1, \xi_2) \mapsto & \left( \begin{array}{c} \nabla_s \xi_0 + \nabla_{\xi_0} {\mc X}_\phi + J (\nabla_t \xi_0 + \nabla_{\xi_0} {\mc X}_{\psi}) + {\mc X}_{\xi_1} + J {\mc X}_{\xi_2}  + \nabla_{\xi_0} \nabla W(u)\\
 \partial_s \xi_2 + [\phi, \xi_2] - \partial_t \xi_1 -[\psi, \xi_1] + d\mu(u) (\xi_0)\\
 \partial_s \xi_1 + [\phi, \xi_1] + \partial_t \xi_2 +[\psi, \xi_2] + d\mu(u) (J \xi_0) \end{array} \right)
\end{split}
\eeq
Then $D_u^W$ is Fredholm if and only if $\hat D_u^W$ is, and they have the same Fredholm index. Let $D_1, D_2, D_3$ be the three components of $\hat D_u^W$. Next, by the asymptotic behavior of affine vortex, there exist $R>0$ and a smooth map $g: {\mb C} \setminus B_R \to K$ such that
\beqn
\lim_{z \to \infty} g(z) u(z) = x_\infty \in \mu^{-1}(0) \cap dW^{-1}(0)
\eeqn 
exists. One can view $x_\infty$ as the value at infinity of the pointlike instanton $u$, even though only its $K$-orbit is well-defined. One can split $T_{x_\infty} V$ as 
\beq\label{eqna5}
T_{x_\infty} V \cong T_X \oplus N_X \oplus G_X
\eeq
where $G_X$ is the subspace generated by infinitesimal $G$-actions (whose dimension equals to the dimension of the complex Lie group $G$), $N_X$ is the normal bundle to the critical locus ${\rm Crit}W$, and $T_X$ (whose dimension equals to the dimension of the classical vacuum $X$) is the orthogonal complement to the first two summands. We can then split 
\beqn
v^* P(TV) \cong \bigoplus_{i=1}^m T_v \oplus N_v \oplus G_v
\eeqn
over the domain ${\mb C}$ whose restriction to the fibre at infinity via the identification $v^* P(TV)|_{\infty} \cong T_{x_\infty} V$ agrees with the splitting \eqref{eqna5}. Moreover, we can identify $G_X$ with the Lie algebra ${\mf g}$ and require that this identification extends to a trivialization of $G_v$ over ${\mb C}$. We can also view $N_X$ as generated by the eigenspaces of $\nabla^2 W$ corresponding to nonzero eigenvalues and require that this identification extends to a trivialization of $N_v$. Then up to compact operators, we may assume that the component $D_1$ of $\hat D_u^W$ preserves the above bundle splitting. The by the same argument of \cite{Venugopalan_Xu}, one can see that the $G_v$-component of $D_1$ together with $D_2$ and $D_3$ is Fredholm with index zero, and the $N_v$-component of $D_1$ also has Fredholm index zero. Then the Fredholm index of $\hat D_u^W$ coincides with that of $D_1$ in the $T_v$-direction, which, by a Riemann--Roch formula (see corresponding part in \cite{Venugopalan_Xu}), gives the formula \eqref{eqna3}.
\end{proof}

Now we need to identify small infinitesimal deformations with nearby gauged maps. Now we need to put a further restriction of the Riemannian metric. We choose a $K$-invariant Riemannian metric on $V$ such that both $dW^{-1}(0) \cap \mu^{-1}(0)$ and $\mu^{-1}(0)$ are totally geodesic. 

Now we define the analogue of the chart \eqref{eqna2}. For $\epsilon>0$ sufficiently small, denote $\tilde {\mc W}_u^\epsilon = \tilde {\mc W}_u \cap \tilde {\mc B}_u^\epsilon$, ${\mc W}_u^\epsilon = {\mc W}_u \cap \tilde {\mc B}_u^\epsilon$. For any $\xi \in \tilde {\mc W}_u^\epsilon$, define the gauged map $u_\xi$ as before. Then define the map 
\beqn
\tilde {\mc G}_u: \tilde {\mc W}_u^\epsilon \to {\mc E}_u,\ \xi \mapsto \tilde {\mc F}_u(\xi) + \left( \begin{array}{c} \nabla W(u_\xi) \\ 0 \end{array}\right).
\eeqn
Here $\tilde {\mc F}_u$ is the map corresponding to the $W = 0$ vortex equation. This is a smooth map whose linearization at $\xi = 0$ is just $\tilde D_u^W$. Then there is a natural map 
\beq\label{eqna6}
\varphi_u^\epsilon: {\mc G}_u^{-1}(0) \cap {\mc W}_u^\epsilon \to \tilde {\mc M}(V),\ \xi \mapsto [u_\xi].
\eeq

\begin{thm}
For $\epsilon>0$ sufficiently small, the map \eqref{eqna6} is a homeomorphism onto an open neighborhood of $[u]$ in $\tilde {\mc M}(V)$. In particular, when $D_u^W$ is surjective, $[u]$ has an open neighborhood in $\tilde {\mc M}(V)$ which is homeomorphic to an open subset of ${\mb R}^{m_d}$.
\end{thm}

\begin{proof}
We first prove $\varphi_u^\epsilon$ is bijective onto a neighborhood of $[u]$. 

\noindent {\it Injectiveness.} Suppose $\varphi_u^\epsilon$ is not injective for all $\epsilon$. Then there exist $u_\xi, u_{\xi'} \in {\mc G}_u^{-1}(0)$ which are gauge equivalent. Then $\xi$ and $\xi'$ are also elements of ${\mc F}_u^{-1}(0)$. By \cite[Proposition 3.3]{Venugopalan_Xu}, when $\epsilon$ is sufficiently small, $\xi = \xi'$. Hence $\varphi_u^\epsilon$ is injective.

\noindent {\it Surjectivenss.} Suppose $\varphi_u^\epsilon$ is not surjective onto a neighborhood of $[u]$. Then there exists a sequence of points $[u_i] \in \tilde {\mc M}(V)$ avoiding the image of $\varphi_u^\epsilon$ converging to $[u]$. Now notice that $[u_i]$ are also elements of $\tilde {\mc N} (V)$ converging to $[u]$. Then for $i$ sufficiently large there exists $\xi_i \in {\mc F}_u^{-1}(0)$ such that $[u_{\xi_i}] = [u_i]$. What we need to show is that $\xi_i\in {\mc G}_u^{-1}(0)$. Notice that $u_{\xi_i}$ solves the gauged Witten equation. Hence it remains to show that $\xi_i \in \tilde {\mc B}_u^W$ and converges to zero, i.e., 
\beq\label{eqna7}
\lim_{i \to \infty} \| \nabla_{\xi_i} \nabla W(u) \|_{L^{p,\delta}} = 0.
\eeq
What we know is
\beq\label{eqna8}
\lim_{i \to \infty} \| \xi_i \|_{1,p,\delta} = 0\Longrightarrow \lim_{i \to \infty} \left( \| d\mu(u)(J\xi_i)\|_{L^{p,\delta}} + \| \xi_i \|_{L^\infty}\right) = 0.
\eeq
Indeed, because $dW^{-1}(0) \cap \mu^{-1}(0)$ is totally geodesic, there exists $C>0$ and a neighborhood $U\subset V$ of $\mu^{-1}(0)$ such that for all $x \in dW^{-1}(0)\cap U$ near $\mu^{-1}(0)$ and $\beta \in T_x V$ with $\exp_x \beta \in dW^{-1}(0)$, there holds
\beqn
| \nabla_\beta \nabla W(x) | \leq C \left( |\beta| |\mu(x)| + |d\mu(x) (J\beta)| \right).
\eeqn
Therefore, for $R>0$ sufficiently large such that $u({\mb C} \setminus B_R) \subset U$, one has
\beqn
\begin{split}
\| \nabla_{\xi_i} \nabla W(u) \|_{L^{p,\delta}} \leq &\ \| \nabla_{\xi_i} \nabla W(u) \|_{L^{p,\delta}(B_R)} + \| \nabla_{\xi_i} \nabla W(u) \|_{L^{p,\delta}({\mb C} \setminus B_R)}\\
\leq &\ C \| \xi_i \|_{L^\infty} + C \| d\mu(u) (J \xi_i) \|_{L^{p,\delta}} + C \| \xi_i \|_{L^\infty} \| \mu(u) \|_{L^{p,\delta}({\mb C} \setminus B_R)}
\end{split} 
\eeqn
Notice that $\mu(u(z))$ decays like $|z|^{-2 + \epsilon}$ (see \cite{Ziltener_Decay}) which has finite $L^{p,\delta}$-norm. Then \eqref{eqna7} follows from the above estimate and \eqref{eqna8}.

Last, notice that ${\mc G}_u^{-1}(0)$ is locally compact as it is a closed subset of a manifold (by adding obstructions, see Lemma \ref{lemmaa8} below) and $\tilde {\mc M}(V)$ is Hausdorff. So the restriction of $\varphi_u^\epsilon$ to a precompact open neighborhood of $0$ in ${\mc G}_u^{-1}(0)$ is a homeomorphism onto its image.
\end{proof}

\subsection{Existence of obstruction spaces}

We first recall the basic results about the existence of ``$K$-immersive points.''
\begin{defn}
Let $\Omega\subset {\mb C}$ be a Riemann surface and $\sigma: \Omega \to (0, +\infty)$ be a smooth function. Let $(u, \phi, \psi)$ be a solution to the vortex equation \eqref{eqn22}. A point $z \in \Omega$ is called a {\bf $K$-immersive point} if $u(z)\in V^{\rm ss}$ and 
\beqn
\partial_s u (z) + {\mc X}_{\phi(z)} (u(z)) \neq 0.
\eeqn
\end{defn}

\begin{lemma}(cf. \cite[Lemma 2.5]{Cieliebak_Gaio_Mundet_Salamon_2002})\label{lemmaa7} If there exists an open subset $U \subset \Omega$ such that $u(U) \subset V^{\rm ss}$ and $\partial_s u + {\mc X}_\phi(u) \equiv 0$ over $U$, then $E(u, \phi, \psi)  = 0$. In particular, if $u$ is a nontrivial pointlike instanton, then the subset of $K$-immersive points is open and dense.
\end{lemma}

\begin{proof}
Notice that points in $V^{\rm ss}$ have trivial or finite $K$-stabilizer. Then the first part of this lemma follows from \cite[Lemma 2.5]{Cieliebak_Gaio_Mundet_Salamon_2002}. If $u$ is a pointlike instanton, then by the fact that pointlike instantons are also affine vortices, the set of $K$-immersive points is nonempty and open. It is dense because $u$ is a holomorphic object and the unstable locus of $V$ is a complex subvariety. 
\end{proof}

\begin{lemma}\label{lemmaa8}
Let $u = (u, \phi, \psi)$ be a nontrivial pointlike instanton and let $U \subset {\mb C}$ be a nonempty open subset consisting of only $K$-immersive points. Then there exists a finite-dimensional vector space $E_u$ and a linear map 
\beqn
\iota: E_u \to {\mc E}_u
\eeqn
satisfying the following conditions. 

\begin{enumerate}

\item $E_u$ is transverse to the image of the operator $D_u$.

\item $\iota(E_u)$ is generated by sections whose supports are contained in $U$.

\end{enumerate}
\end{lemma}

\begin{proof}
Let $\hat D_u^W$ be the augmented linearized operator at $u$ (see \eqref{aug}). Then $\hat D_u$ is  a Fredholm operator. Suppose $D_u^W$ is not surjective. Then $\hat D_u^W$ is not surjective as well and hence has a finite-dimensional cokernel. Via the $L^2$-pairing, the dual space of $L^{p,\delta}$ is $L^{p', -\delta}$ where $\frac{1}{p} + \frac{1}{p'} = 1$. Then each element of the cokernel of $\hat D_u^W$ can be identified with an element $\eta \in L^{p', -\delta}({\mb C}, u^* TV \oplus {\mf k} \oplus {\mf k})$ which is orthogonal to the image of $\hat D_u^W$. Then $\eta$ is annihilated by the formal adjoint of the linearization $\hat D_u^W$, which is a first order linear operator with smooth coefficients. Then by elliptic regularity, $\eta$ is smooth.

We write $\eta = (\eta_0, \eta_1, \eta_2)$ where $\eta_0$ is a vector field along $u$. We claim that $\eta_0 \neq 0$. Indeed, if $\eta_0 = 0$, then using test functions of the form $\xi = (\xi_0, 0, 0)$, one has 
\beqn
\langle \hat D_u^W (\xi), \eta \rangle = \langle d\mu(u) (\xi_0), \eta_1 \rangle + \langle d\mu(u)(J \xi_0), \eta_2 \rangle = 0.
\eeqn
As for a nonempty open subset $U$ of $K$-immersive points we have $u(U) \subset V^{\rm ss}$, one knows the map $\xi_0 \mapsto (d\mu(u)(\xi_0), d\mu(u)(J\xi_0))$ is surjective, it follows that $\eta_1|_U \equiv \eta_2|_U \equiv 0$. Then $\eta$ vanishes on $U$. By the unique continuation principle, $\eta \equiv 0$, which is a contradiction. 

Choose a suitable cut-off function $\beta$ supported in the given open subset $U$. Then $(\beta \eta_0, 0, 0)$ is contained in $\tilde {\mc E}_u$ but not orthogonal to $\eta$. Hence $(\beta \eta_0, 0, 0)$ is not in the image of $\hat D_u^W$. Equivalently, $(\beta \eta_0, 0)$ is not in the image of $D_u^W$, which is the restriction of $\hat D_u^W$ to the Coulomb slice. As the cokernel is finite-dimensional, inductively one can find a list of generators for a finite-dimensional obstruction space $E_u$ spanned by elements of the form $(\beta \eta_0, 0)$ satisfying the requirement.
\end{proof}

\bibliography{../mathref}

\bibliographystyle{amsalpha}

\end{document}